\documentclass{amsart}
\usepackage{url}
\usepackage[utf8]{inputenc}
\usepackage{verbatim}
\usepackage{paralist}
\usepackage[scaled]{helvet}

\theoremstyle{plain}
\newtheorem{theorem}{Theorem} 
\newtheorem{thm}{Theorem}[subsection] 
\newtheorem{proposition}[thm]{Proposition}

\theoremstyle{definition}
\newtheorem{definition}[thm]{Definition}

\theoremstyle{remark}
\newtheorem{remark}[thm]{Remark}
\numberwithin{equation}{section}

\newcommand{\add}{\operatorname{add}}
\newcommand{\End}{\operatorname{End}}
\newcommand{\supp}{\operatorname{supp}}
\newcommand{\indice}{\operatorname{index}}
\newcommand{\Hom}{\operatorname{Hom}}
\newcommand{\Ext}{\operatorname{Ext}}
\newcommand{\coh}{\operatorname{coh}}
\newcommand{\rad}{\operatorname{rad}}
\newcommand{\Lmod}[1]{#1\!\operatorname{-mod}}
\newcommand{\gld}{\operatorname{gl.dim}}
\newcommand{\ddim}{\underline{\dim}}
\newcommand{\obj}{\operatorname{obj}}

\newcommand{\opp}{{\mathrm{op}}}  

\newcommand{\ZZ}{\mathbb{Z}}
\newcommand{\QQ}{\mathbb{Q}}
\newcommand{\T}{\mathbb{T}}
\newcommand{\XX}{\mathbb{X}}

\newcommand{\cC}{\mathcal{C}}
\newcommand{\cJ}{\mathcal{J}}
\newcommand{\cP}{\mathcal{P}}
\newcommand{\cQ}{\mathcal{Q}}
\newcommand{\cT}{\mathcal{T}}

\newcommand{\CQ}[1]{K\langle\hspace{-0.05cm}\langle #1\rangle\hspace{-0.05cm}\rangle}
\newcommand{\bil}[1]{\langle #1\rangle}   
\newcommand{\ebrace}[1]{\langle #1 \rangle} 

\newcommand{\Lam}{\Lambda}
\newcommand{\gam}{\gamma}
\newcommand{\lam}{\lambda}

\begin{document}

\title{Tubular Jacobian Algebras}
\author{Christof Geiss}
\address{Instituto de Matemáticas, UNAM, Ciudad Universitaria, C.P. 
04510 México D.F., MEXICO}
\email{christof@matem.unam.mx}
\thanks{Both authors acknowledge partial support from CONACYT Grant No.  81948}

\author{Raúl González-Silva}
\email{rulo65@ciencias.unam.mx}
\thanks{R.G. was supported by a Ph.D. grant from CONACYT}


\begin{abstract}
We show that the endomorphism ring of each cluster tilting object in a tubular 
cluster category is a finite dimensional Jacobian algebra which is tame of 
polynomial growth. Moreover, these Jacobian algebras are given by a quiver
with a non-degenerate potential and mutation of cluster tilting objects is
compatible with mutation of QPs.
\end{abstract}

\maketitle

\section{Introduction}
 Tubular cluster algebras were introduced in \cite{BACH} as a particular 
class of mutation finite cluster  algebras. Their common feature is that they 
admit an additive categorification by a tubular cluster category of the
corresponding type.   
Recall  that tubular cluster categories are by definition of the form 
$\mathcal{C}_{\XX}=\mathcal{D}^{b}(\coh(\XX))/\ebrace{\tau^{-1}[1]}$, 
where $\coh{\XX}$ is the category of coherent sheaves on a weighted projective 
line of tubular weight type i.e. 
$(2,2,2,2;\lambda), (3,3,3), (4,4,2)$ and $(6,3,2)$, see~\cite{GeLe}.
It follows from~\cite{Ke} that $\mathcal{C}_{\mathbb{X}}$ is a triangulated 
$2$-Calabi-Yau category which admits a cluster structure~\cite{BaKuLe}.
Moreover, in these cluster categories  the indecomposable
rigid objects are in bijection with the positive real Schur roots of the
corresponding elliptic root system. Via the cluster character indecomposable 
rigid objects are in bijection with cluster variables~\cite{BACH}.

It is not hard to derive from known (non-trivial) results that in this 
situation the endomorphism ring of each cluster tilting object can be
described as Jacobian algebra via an rather explicit quiver with potential, 
and the mutation of cluster tilting objects is compatible with the mutation of 
QPs in the sense of~\cite{DWZ1}. See Section~\ref{ssec:TubJac} for precise 
statements and proofs. In particular, the fact that for tubular cluster 
categories the exchange graph of cluster tilting objects is 
connected~\cite[Prp.~8.6]{BaKuLe}, makes the statements~(b) and~(c) of 
Proposition~\ref{prp:tubjac} straightforward and convenient.

Finally, recall from~\cite[Section~2.1]{KeRe} that for a cluster tilting object
$T\in\cC_\XX$ we have 
$\Lmod{\End_\cC(T)^\opp}\cong \cC/\add{T}$. 
In particular, the Auslander-Reiten quiver of these algebras consists just of
tubular families of the corresponding weight type, with all but finitely many
tubes being stable. 

Since the tubular cluster categories are orbit categories of derived tame 
categories in the sense of~\cite{GeKr}, this suggests strongly that all these 
algebras are tame.  In fact, this is our main result:

\begin{theorem} \label{thm:main}
Let $K$ be an algebraically closed field, $\cC_{\XX}$
a tubular cluster category over $K$ as above and $T\in\cC_{\XX}$
a basic cluster tilting object, then the endomorphism ring
$\End_{\cC_{\XX}}(T)$ is a finite dimensional Jacobian algebra which is
tame of polynomial growth.
\end{theorem}

Our strategy is as follows:  by Proposition~\ref{prp:tubjac}
all the tubular Jacobian algebras of a given tubular type
are given explicitly by non-degenerate QPs and  connected via QP-mutations. 
Since the representation type of Jacobian algebras is 
preserved under mutation~\cite{GeLaSc}, it is sufficient to show for each 
tubular type that the endomorphism ring of a single cluster 
tilting object, is tame (of polynomial growth).

We display in Figure~\ref{f:0} the list of representatives for which
we are going to show tameness explicitly. More precisely,
$(Q^{(1)},W^{(1)})$ corresponds to the tubular type $(3,3,3)$;
$(Q^{(2)}, W^{(2)}_\lam)$ corresponds to the tubular type $(2,2,2,2;\lam)$, 
$(Q^{(3)}, W^{(3)})$ corresponds to tubular type $(4,4,2)$ and
$(Q^{(4)}, W^{(4)})$ corresponds to tubular type $(6,3,2)$. 

For each of these representatives we provide
a Galois covering coming from the natural $\mathbb{Z}$-grading, which turns
out to be iterated tubular in the sense of de la Peña-Tomé.
This provides in each case a quite explicit description of the 
indecomposables of this covering, see Section~\ref{ssec:TubIt}. Since iterated
tubular algebras are locally support finite, we obtain also a precise
description of the indecomposables of our tubular Jacobian algebra by
a well-known result of Dowbor and Skowroński, see~Theorem~\ref{DS}.

In fact, we see from the description in Proposition~\ref{prp:tubjac}~(a) 
that all tubular Jacobian algebras admit a Galois covering without
oriented cycles. We conjecture that this covering is always iterated tubular. 
This is easy to verify directly for
the tubular type $(2,2,2,2;\lam)$ since there are only four families of 
Jacobian algebras of this type. See also Remark~\ref{rem:tubjac}~(2) for
more information about this slightly unusual case.
However, for type $(6,3,2)$ there are a priory several thousand cases
to be analyzed.

    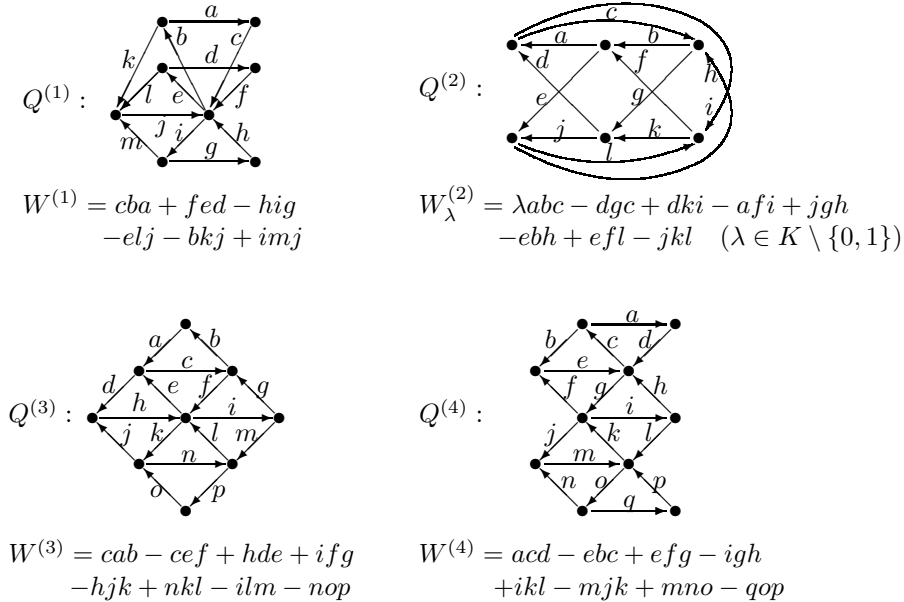
\begin{figure}[ht]
      \setlength{\unitlength}{0.62mm}
      \begin{picture}(120,140)
        \put(60,80){
          \put(-35,45){\circle*{2}} 
          \put(-55,45){\circle*{2}}
          \put(-35,35){\circle*{2}}
          \put(-55,35){\circle*{2}}
          \put(-35,15){\circle*{2}}
          \put(-55,15){\circle*{2}}
          \put(-45,25){\circle*{2}}
          \put(-65,25){\circle*{2}}
          \multiput(-36,44)(-20,0){2}{\vector(-1,-2){8.7}}
          \multiput(-36,34)(-20,0){2}{\vector(-1,-1){8}}
          \multiput(-36,16)(-10,10){2}{\vector(-1,1){8}}
          \multiput(-46,24)(-10,10){2}{\vector(-1,-1){8}}
          \multiput(-56,16)(-20,0){1}{\vector(-1,1){8}}
          \multiput(-45.8,26)(-20,0){1}{\vector(-1,2){8.8}}
          \multiput(-53.5,45)(0,-10){2}{\vector(1,0){17}}
          \multiput(-63.5,25)(10,-10){2}{\vector(1,0){17}}  
          \put(-46,46){$a$}
          \put(-52.2,40){$b$}
          \put(-40.5,40){$c$}
          \put(-46,36){$d$}
          \put(-53,27.5){$e$}
          \put(-39.6,27.5){$f$}
          \put(-46,16.7){$g$}
          \put(-39.4,19.5){$h$}
          \put(-52.4,19.5){$i$}
          \put(-56.5,21.3){$j$}
          \put(-63.8,35){$k$}
          \put(-58.8,27.5){$l$}            
          \put(-64.2,17.6){$m$}            
          \put(-85,4){$W^{(1)}=cba+fed-hig$}
          \put(-68,-3){$-elj-bkj+imj$}
          \put(-85,26.4){$Q^{(1)}:$}

          \put(60,40){\circle*{2}} 
          \put(60,20){\circle*{2}}
          \put(40,40){\circle*{2}}
          \put(40,20){\circle*{2}}
          \put(20,40){\circle*{2}}
          \put(20,20){\circle*{2}}
          \multiput(58,40)(-20,0){2}{\vector(-1,0){16}}
          \multiput(58,20)(-20,0){2}{\vector(-1,0){16}}
          \multiput(58.5,38.5)(-20,0){2}{\vector(-1,-1){17}}
          \multiput(58.5,21.5)(-20,0){2}{\vector(-1,1){17}}
          \qbezier(21.5,41.4)(40,49)(58.5,41.4)
          \put(53.4,43.1){\vector(3,-1){5.7}}
          \qbezier(20.5,42)(42.5,54)(60.5,45)
          \qbezier(60.5,45)(73,37)(61.5,21.8)
          \put(65,27){\vector(-2,-3){3.6}}
          \qbezier(21.5,18.6)(40,11)(58.5,18.6)
          \put(53.4,16.8){\vector(3,1){5.7}}
          \qbezier(20.5,18)(42.5,6)(60.5,15)
          \qbezier(60.5,15)(73,23)(61.5,38.2)
          \put(65,33){\vector(-2,3){3.6}}
          \put(29,40.3){$a$}
          \put(49,40.3){$b$}
          \put(40,45.5){$c$}
          \put(25,35){$d$}
          \put(25,27.2){$e$}
          \put(46,35){$f$}
          \put(45.5,28.2){$g$}
          \put(61,32){$h$}
          \put(61.3,25){$i$}
          \put(29,20.3){$j$}
          \put(49,20.3){$k$}
          \put(40,15.2){$l$}
          \put(0,29){$Q^{(2)}:$}
          \put(0,4){$W^{(2)}_\lam=\lambda abc-dgc+dki-afi+jgh$}
          \put(16.8,-3){$-ebh+efl-jkl\quad(\lambda\in K\setminus\{0,1\})$}

          \put(-50,-40){\circle*{2}}
          \put(-70,-40){\circle*{2}}
          \put(-30,-40){\circle*{2}}
          \put(-40,-30){\circle*{2}}
          \put(-60,-30){\circle*{2}}
          \put(-50,-20){\circle*{2}}
          \put(-40,-50){\circle*{2}}
          \put(-60,-50){\circle*{2}}
          \put(-50,-60){\circle*{2}}
          \multiput(-51,-39)(10,10){2}{\vector(-1,1){8}}
          \multiput(-41,-49)(10,10){2}{\vector(-1,1){8}}
          \multiput(-61,-49)(10,-10){2}{\vector(-1,1){8}}
          \multiput(-61,-31)(10,10){2}{\vector(-1,-1){8}}
          \multiput(-51,-41)(10,10){2}{\vector(-1,-1){8}} 
          \multiput(-41,-51)(10,10){2}{\vector(-1,-1){8}}
          \multiput(-68.4,-40)(20,0){2}{\vector(1,0){17}}
          \multiput(-58.4,-50)(0,20){2}{\vector(1,0){17}}
          \put(-58,-24.8){$a$}
          \put(-45,-24.8){$b$}
          \put(-51,-29.4){$c$}
          \put(-68,-34.8){$d$}
          \put(-54,-34.8){$e$}
          \put(-47.5,-34.8){$f$}
          \put(-34.9,-34.8){$g$}
          \put(-61.6,-38.6){$h$}
          \put(-41,-39.5){$i$}
          \put(-64,-45){$j$}
          \put(-57.7,-45){$k$}
          \put(-44.6,-45){$l$}
          \put(-39.4,-45){$m$}
          \put(-51,-49.4){$n$}
          \put(-57.7,-57){$o$}
          \put(-44.4,-57){$p$}
          \put(-88,-71){$W^{(3)}=cab-cef+hde+ifg$}
          \put(-75,-78){$-hjk+nkl-ilm-nop$}
          \put(-88,-41){$Q^{(3)}:$}
          \put(35,-40){\circle*{2}}
          \put(55,-20){\circle*{2}}
          \put(55,-40){\circle*{2}}
          \put(45,-30){\circle*{2}}
          \put(25,-30){\circle*{2}}
          \put(35,-20){\circle*{2}}
          \put(45,-50){\circle*{2}}
          \put(25,-50){\circle*{2}}
          \put(35,-60){\circle*{2}}
          \put(55,-60){\circle*{2}}
          \multiput(34,-41)(0,20){2}{\vector(-1,-1){8}}
          \multiput(44,-51)(0,20){2}{\vector(-1,-1){8}}
          \multiput(54,-41)(0,20){2}{\vector(-1,-1){8}}
          \multiput(34,-39)(0,-20){2}{\vector(-1,1){8}}
          \multiput(44,-49)(0,20){2}{\vector(-1,1){8}} 
          \multiput(54,-39)(0,-20){2}{\vector(-1,1){8}}
          \multiput(37,-40)(0,20){2}{\vector(1,0){16}}
          \multiput(27,-50)(0,20){2}{\vector(1,0){16}}
          \multiput(37,-60)(0,0){2}{\vector(1,0){16}}
          \put(44.3,-19.4){$a$}
          \put(27,-25){$b$}
          \put(40.2,-25){$c$}
          \put(47,-25){$d$}
          \put(33.8,-29.4){$e$}
          \put(30.8,-35){$f$}
          \put(37.5,-35){$g$}
          \put(50.1,-35){$h$}
          \put(44.3,-39.2){$i$}
          \put(27,-45){$j$}
          \put(40.2,-45){$k$}
          \put(48.4,-45){$l$}
          \put(33,-49.3){$m$}
          \put(30.4,-55){$n$}
          \put(37.6,-55){$o$}
          \put(50.1,-55){$p$}
          \put(43.8,-59.2){$q$}
          \put(0,-71){$W^{(4)}=acd-ebc+efg-igh$}
          \put(16,-78){$+ikl-mjk+mno-qop$}

          \put(0,-41){$Q^{(4)}:$}
        }     
      \end{picture}
\caption{Representatives of tubular QPs}
\label{f:0}
    \end{figure}

\section{Preliminaries}
\subsection{Grothendieck group}
Let $K$ be an algebraically closed field.
\begin{definition}
Let $A$ be a tame (representation-infinite) connected and hereditary 
$K$-algebra and $_AT$ a preprojective tilting module. 
The algebra $B:=\End(_AT)$ is called a tame concealed algebra.
\end{definition}
    
Let $B$ be a tame concealed algebra, $K_0(B)$ its Grothendieck group and $C_B$ 
its Cartan matrix~\cite[2.2.4]{RIN}. $K_0(B)$ is a free abelian group
which has a natural basis consisting of the classes of the simple $B$-modules
$S_1, S_2,\ldots S_n$. Via this basis we identify $K_0(B)$ with $\ZZ^n$.
Consequently, we identify the dimension vector $\ddim M\in\ZZ^n$ of a finite 
dimensional $B$-module $M$ with its class in $K_0(B)$. 
For typographic reasons we interpret the elements of $\ZZ^n$ as row vectors.
In this setting, 
the Cartan matrix $C_B$ is invertible and we consider Ringel's bilinear form 
$\bil{\_ \phantom{,},\_}$ on $K_0(B)$ given by
\[
\bil{x,y}=xC_B^{-T}y^T.
\]
Recall that we have
$\bil{\ddim M,\ddim N} = \sum_{i=0}^2(-1)^i\dim \Ext^i_B(M,N)$ for
all finite dimensional $B$-modules $M$ and $N$ since $\gld(B)\leq 2$.
We denote by  $\chi_B$ the quadratic form given by 
$\chi_B(x)=\bil{x,x}$. It is well known that $\chi_B$ is non-negative. 
Thus the radical of $\chi_B$ is the subgroup $\{r\in K_0(B)\mid \chi_B(r)=0\}$.
It is well-known that the $\chi_B$ has corank 1,  i.e. the radical of $\chi_B$ 
has rank $1$.

For our purpose an indecomposable $B$-module $M$ is \emph{regular}, if
$\tau^i M\neq 0$ for all $i\in\ZZ$.

Finally $\Phi:=-C_B^{-T}C_B$ is the Coxeter transformation  on $K_0(B)$.

\begin{remark}
Since $\gld(B)\leq 2$, the matrix $C_B^{-T}$ codifies arrows and minimal 
relations for $B$, i.e. $(C_B^{-T})_{i,j}$ is equal to the number of relations 
starting at $j$ and ending at $i$ minus the number of arrows starting at
$j$ and ending at $i$.
\end{remark}
    
 We collect from~\cite{RIN} the following well-known results:

 \bigskip
 \begin{proposition}\label{ringel}
 Let $B$ be a tame concealed algebra, then:
 \begin{asparaenum}[a)]                   
\item 
There exists unique positive generator $h$ of the radical of $\chi_B$, i.e. 
$h=\ddim R$ for $R$ an indecomposable (regular) $B$-module and 
$\rad \chi_B=\ZZ h$. Then an
indecomposable  $B$-module $M$ is regular if and only if 
$\bil{h,\ddim M}=0$.
\item 
$\ddim(\tau M)=(\ddim M)\Phi$, for any indecomposable regular $B$-module, 
where $\tau$ is the AR-translate. 
\item 
An indecomposable regular $B$-module $M$ is simple regular if\newline 
$\displaystyle{\sum_{i=0}^{m-1}}(\ddim\tau^iM)=h$, 
where $m$ is the $\tau$-period of $M$.
\end{asparaenum}      
\end{proposition}

For $\emph{(a)}$  see~\cite[Thm. 4.3 (3)]{RIN}. 
$\emph{(b)}$ follows from \cite[2.4 (4)]{RIN}
since the regular modules over a tame concealed algebra have projective 
dimension 1 \cite[3.1 (5)]{RIN}.  
(c) follows with (b) by tilting from the corresponding statement about 
regular modules over tame hereditary algebras. In this case it can be verified 
by direct inspection, see for example \cite[sec. 6]{DRIN}.

\subsection{Tubular and iterated tubular algebras} \label{ssec:TubIt}
A \emph{tubular algebra} is a tubular extension of a tame concealed algebra of 
extension type $(2,2,2,2), (3,3,3), (4,4,2)$ or $(6,3,2)$ in the sense 
of~\cite[page 230]{RIN}. Note that the global dimension of a tubular algebra 
is always $2$, see~\cite{RIN}.

Let $\mathbb{T}\in\{(2,2,2,2),(3,3,3),(4,4,2),(6,3,2)\}$ and let $A$ be a 
tubular algebra of tubular type $\T$. By definition, it can be viewed as an 
extension of a tame concealed algebra $A_0$ and also as a coextension of a tame
concealed algebra $A_{\infty}$ \cite[pages 268-269]{RIN}.
    
Let $h_0$, respectively $h_{\infty}$, be the positive radical generator of 
$K_0(A_0)$, res\-pectively of $K_0(A_{\infty})$. If $M$ is an indecomposable 
$A$-module, we define:
\[
\indice(M)=-\frac{\bil{h_0,\ddim M}_A}{\bil{h_{\infty},\ddim M}_A}
\in\QQ\cup\{\infty\}
\]
For $\gamma\in\QQ_{\,>0}$, let $\mathcal{T}_{\gamma}$ be the module class 
given by all the indecompo\-sable $A$-modules with index $\gamma$. 
The module class $\mathcal{T}_{\gamma}$ is a sincere stable 
tubular $\mathbb{P}^1(k)$-family of type $\T$. \cite[Thm. 5.2 (2)]{RIN}

The following theorem~\cite[Thm. 5.2 (4)]{RIN} describes the structure 
of $\Lmod{A}$.

\begin{thm}[Ringel]
Let $A$ be a tubular algebra of type $\T$. Then $\Lmod{A}$ has the following 
components: a preprojective component
$\mathcal{P}_0$ (the same as for $A_0$), 
for each $\gamma\in\QQ_{\,\geq 0}\cup\{\infty\}$ a 
separating (\cite[sec. 3.1]{RIN}) tubular $\mathbb{P}^1(k)$-family 
$\mathcal{T}_\gamma$, all but $\mathcal{T}_0$ and $\mathcal{T}_\infty$ being 
stable of type $\T$, and a preinjective component $\mathcal{Q}_\infty$.

The following figure helps us to visualize the structure of $\Lmod{A}$:
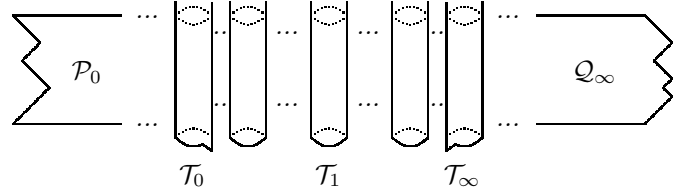
      \begin{center}
        \setlength{\unitlength}{0.6mm}
        \begin{picture}(200,50)
          \put(100,26){     
            \put(-46,12){\line(-1,0){24}}      
            \qbezier(-70,12)(-67,9)(-64,6)
            \qbezier(-68,2)(-66,4)(-64,6)
            \qbezier(-68,2)(-65,-1)(-62,-4)
            \qbezier(-70,-12)(-66,-8)(-62,-4)      
            \put(-46,-12){\line(-1,0){24}}
            \put(-57,-2){$\mathcal{P}_0$}

            \put(-43,11.5){...}
            \put(-43,-12.4){...}            

            \put(-34,-15){\line(0,1){30}}
            \put(-26,-18){\line(0,1){33}}
            \qbezier[10](-26,12)(-30,16)(-34,12)
            \qbezier[10](-26,12)(-30,8)(-34,12)      
            \qbezier[10](-26,-15)(-30,-11)(-34,-15)
            \qbezier(-28,-16.5)(-27,-17.25)(-26,-18)      
            \qbezier(-28,-16.5)(-31,-18)(-34,-15)      
            \put(-33.2,-25){$\mathcal{T}_0$}
            
            \put(-25.9,8){.}
            \put(-24.2,8){.}
            \put(-25.9,-8){.}
            \put(-24.2,-8){.}

            \put(-22,-15){\line(0,1){30}}
            \put(-14,-15){\line(0,1){30}}
            \qbezier[10](-22,12)(-18,16)(-14,12)
            \qbezier[10](-22,12)(-18,8)(-14,12)      
            \qbezier[10](-22,-15)(-18,-11)(-14,-15)
            \qbezier(-22,-15)(-18,-19)(-14,-15)            

            \put(-12,8){...}
            \put(-12,-8){...}

            \put(-4,-15){\line(0,1){30}}
            \put(4,-15){\line(0,1){30}}
            \qbezier[10](-4,12)(0,16)(4,12)
            \qbezier[10](-4,12)(0,8)(4,12)      
            \qbezier[10](-4,-15)(0,-11)(4,-15)
            \qbezier(-4,-15)(0,-19)(4,-15)      
            \put(-3,-25){$\mathcal{T}_1$}

            \put(6,8){...}
            \put(6,-8){...}

            \put(14,-15){\line(0,1){30}}
            \put(22,-15){\line(0,1){30}}
            \qbezier[10](22,12)(18,16)(14,12)
            \qbezier[10](22,12)(18,8)(14,12)      
            \qbezier[10](22,-15)(18,-11)(14,-15)
            \qbezier(22,-15)(18,-19)(14,-15)

            \put(23.9,8){.}
            \put(22.2,8){.}
            \put(23.9,-8){.}
            \put(22.2,-8){.}

            \put(34,-15){\line(0,1){30}}
            \put(26,-18){\line(0,1){33}}
            \qbezier[10](26,12)(30,16)(34,12)
            \qbezier[10](26,12)(30,8)(34,12)      
            \qbezier[10](26,-15)(30,-11)(34,-15)
            \qbezier(28,-16.5)(27,-17.25)(26,-18)      
            \qbezier(28,-16.5)(31,-18)(34,-15)      
            \put(25.5,-25){$\mathcal{T}_\infty$}

            \put(37,11.5){...}
            \put(37,-12.4){...}

            \put(46,12){\line(1,0){24}}      
            \qbezier(70,12)(73,9)(76,6)
            \qbezier(72,2)(74,4)(76,6)
            \qbezier(72,2)(74,0)(76,-2)
            \qbezier(74,-4)(75,-3)(76,-2)
            \qbezier(74,-4)(75,-5)(76,-6)
            \qbezier(70,-12)(73,-9)(76,-6)      
            \put(46,-12){\line(1,0){24}}
            \put(54,-2){$\mathcal{Q}_\infty$}      
          }
        \end{picture}
      \end{center}
Here, non-zero maps between different classes go only from left to right; 
given indecomposable modules $X$ and $Y$ with $\Hom(X,Y)\not=0$, then 
either $X$ and $Y$ belong to the same component, or  $X\in\mathcal{P}_0$,  or 
$X\in\mathcal{T}_\gamma$ and $Y\in\mathcal{T}_\delta$ and $\gamma<\delta$,
or $Y\in\mathcal{Q}_\infty$.
\end{thm}

Let $\Lam= K Q /I$ a basic finite-dimensional $K$-algebra, 
presented as a path algebra modulo an admissible ideal.  By some abuse of
notation we say that $\Lam'\subset \Lam$ is a \emph{convex subalgebra} if 
for some convex subset $I\subset Q_o$ and $e=\sum_{i\in I} e_i$ we have
$\Lam'=e\Lam e$. Suppose that $\Lam$ has convex subalgebras
$\Lam_1, \Lam_2, \ldots, \Lam_m$ and $A_2, A_3, \ldots, A_m$ such that
the following holds:
\begin{itemize}
\item 
$\Lam_i$ is a tubular algebra for $i=1,2,\ldots, m$,
\item
$A_i\subset \Lam_{i-1}\cap \Lam_i$ is a tame concealed algebra for $i=2,\ldots,m$,
\item
$\Lam_{i-1}= {_{j=1}^{\ s_i}[E^-_{i,j}]A_i}$ is a tubular co-extension and
$\Lam_i =A_i[E^+_{i,j}]_{j=1}^{t_i}$ is a tubular extension for $i=2,\ldots,m$,
\item
$\Lam=\Lam_1+\Lam_2+\cdots+\Lam_m$.
\item 
If $R$ is an indecomposable projective or indecomposable $\Lam$-module,
then $\supp(R)$ is contained in some $\Lam_i$ if $R$ is not projective-injective,
otherwise $\supp(R/\operatorname{soc}(R))$ is contained in some $\Lam_i$.
\end{itemize}
Then, following de la Peña and Tomé~\cite{DELA}, $\Lam$ is called an 
\emph{iterated tubular algebra}. 
Let in this situation $A_0\subset \Lam_1$ and $A_{m+1}\subset\Lam_m$
be convex, tame concealed algebras such that $\Lam_1=A_0[E^+_{0,j}]_{j=1}^{t_0}$
is a tubular extension and $\Lam_m= {_{\phantom{x} j=1}^{s_{m+1}}[E^-_{m+1,j}]A_{m+1}}$ a 
cotubular extension. From
the considerations in~\cite[Sec.~2]{DELA} we obtain the following description
of $\Lmod{(\Lam/S)}$, where $S\subset \Lam$ is the ideal which is spanned
by the socles of indecomposable projective-injective modules. Note that
that the preprojective injective modules are the only indecomposable modules 
which are not killed by $S$. 
\begin{multline*}
\Lmod{(\Lam/S)} =
\cP_0\vee \cT_{0}[E^+_{0,j}]_{j=1}^{t_0} \vee\\ 
\bigvee_{i=1}^m \left(\bigvee_{\gam\in\QQ_{\,>0}}\! \cT_{i,\gam} 
\vee\, {_{j=1}^{\phantom{jx} s_i}[E^-_{i,j}] \cT_i[E^+_{i,j}]_{j=1}^{t_i}}\!\right) 
\vee\, {_{\phantom{xj} j=1}^{\ s_{m+1}}[E^-_{m+1,j}]\cT_{m+1}}\vee \cQ_{m+1},
\end{multline*}
where 
\begin{itemize}
\item
$\cP_0$ is the preprojective component of $A_0$, 
\item
$\cT_i$ is the separating tubular family of $A_i$ for $i=0,1,\ldots,m+1$,
\item
$\cT_{i,\gam}$ is the separating tubular family of index $\gam$ of $\Lam_i$
for $i=1,2,\ldots,m$ and
\item
$\cQ_{m+1}$ is the preinjective component for $A_{m+1}$.
\end{itemize}
In particular, $\Lam$ is tame of polynomial growth and the support of
any indecomposable module is contained in $\Lam_i+\Lam_{i+1}$ for some $i$.
Note, that our definition is slightly more restrictive than the one 
proposed by de la Peña and Tomé in that we avoid branch extensions.
This definition is sufficient for our purpose.

\subsection{Galois Coverings}
In this section we use the functorial approach to representation theory of 
algebras~\cite{Gabriel}. 

We consider algebras as locally bounded  
$K$-categories~\cite[Sec.~1.1]{Gabriel}. Thus, if $B$ is a finite dimensional
basic $K$-algebra with $1_B=e_1+\cdots +e_n$ a decomposition of the $1_B$ into
primitive orthogonal idempotents, we identify $B$ with a category 
with objects $\{1,2,\ldots,n\}$ and morphism spaces $B(i,j):= e_j B e_i$.
Under this identification, $B$-left modules correspond to $K$-linear 
functors from $B$ to the category of $K$-vector spaces.

Given a locally bounded  $K$-category  $A$ and a group $G$ acting freely on 
$A$ by $K$-linear automorphisms we denote by $A/G$  the orbit category. 
In this setting the canonical projection $F:A\to A/G$ is by definition a
\emph{Galois covering}~\cite[Sec.~3]{Gabriel}.  
Associated to $F$ we have the ``push-down" functor 
$F_{\lambda}:\Lmod{A}\to\Lmod{(A/G)}$, given by
\[
    (F_{\lambda}M)(a)=\coprod_{x\in F^{-1}(a)}M(x)
\]
for each object $a$ in $A/G$ and the obvious action of morphisms.

\begin{remark} \label{rem:gal}
For us is the following, well-known construction of Galois coverings  important:
Suppose that the  locally bounded basic $K$-category
$B$ is $\ZZ$-graded, i.e.~we have $B(x,y)=\oplus_{k\in\ZZ} B(x,y)_k$ and
for $f\in B(x,y)_k$ and $g\in B(y,z)_l$ we have $g\circ f\in B(x,z)_{k+l}$.
Then we can define a category $\tilde{B}$ with $\obj(\tilde{B})=\obj(B)\times\ZZ$
and $\tilde{B}((x,k),(y,l)):= B(x,y)_{l-k}$. The group $\ZZ$ acts freely
on $\tilde{B}$ by translations.  On objects of $\tilde{B}$ the action is
given by $g\cdot (x,k):= (x,g+k)$. Clearly we can identify the orbit category
$\tilde{B}/\ZZ$ with $B$. The canonical projection is then 
identified with
the functor $F\colon\tilde{B}\rightarrow B$, which sends $(x,k)$ to $x$
and $\tilde{B}((x,k), (y,l))$ to $B(x,y)_{l-k}\subset B(x,y)$.
\end{remark}

 For any $A$-module $M$, we will denote by $\supp M$ the support of $M$, that 
is, the full subcategory of $A$ formed by all objects $x$ such that 
$M(x)\not=0$. For each object $x$ in $A$, denote by $A_x$ the full 
subcategory of $A$ consisting of all the objects in $\supp M$, where $M$ is 
any indecomposable module with $x$ in $\supp M$.  We say
that $A$ is \emph{locally support-finite} if for every object $x$ in $A$,  
the number of objects in $A_x$ is finite.

From \cite{Dosko} we have the following:
    
\begin{thm}[Dowbor-Skowroński]\label{DS}
Let $A$ be a locally support-finite $K$-cate\-gory and let $G$ be a free 
abelian group acting freely on $A$. Assume that the action of $G$ on the 
isoclasses of indecomposable finitely generated $A$-modules is free. 
Then $A/G$ is locally support-finite and the push-down functor 
$F_{\lambda}:\Lmod{A}\to \Lmod{(A/G)}$ induces a bijection between
the $G$-orbits of isoclasses of indecomposable objects in $\Lmod{A}$ and the 
isoclasses of indecomposable objects in $\Lmod{(A/G)}$. In particular, 
$A$ is representation-tame if and only if $A/G$ is so.
\end{thm}  

\begin{remark}
The theorem in~\cite{Dosko} is more general than the theorem above. 
We state it here in this form which is sufficient for our work.
\end{remark}

\subsection{Jacobian algebras}
In this subsection, we briefly recall the definition of Jacobian algebras. 
For more details see \cite{DWZ1}.

Let $Q$ be a quiver without loops and oriented $2$-cycles, 
and denote by $\CQ{Q}$ the complete path algebra of the quiver $Q$. Let 
$\CQ{Q}_{cyc}$ be the completion of the subspace spanned by all the oriented 
cycles in $Q$.

 A {\it{potential}} in $Q$ is an element $W\in\CQ{Q}_{cyc}$, i.e. a possibly 
infinite linear combination of cycles.

For any arrow $a\in Q_1$, define a continuous linear map 
$\partial_a\colon\CQ{Q}_{cyc}\rightarrow\CQ{Q}$, 
called the {\it{cyclic derivative}}, which acts on oriented cycles by
\[
 \partial_a(a_1\cdot\dots\cdot a_d)=\sum_{p:a_p=a}a_{p+1}\dots a_da_1\dots a_{p-1}
\]
Given a potential $W$ in $Q$, the {\it{Jacobian ideal}} $\cJ(W)$ of $W$ 
is the closure of the ideal in $\CQ{Q}$ which is generated by the elements 
$\partial_a(W)$ with $a\in Q$. 
The {\it{Jacobian algebra}}, denoted by $\cP(Q,W)$, is the factor 
algebra $\CQ{Q}/\cJ(W)$.

\begin{remark}~\label{rem:fjac}
Let $Q$ be a quiver without oriented $2$-cycles, and 
$W$ a  potential for $Q$ which is \emph{finite} linear
combination of cycles. Thus, for each arrow $a$ of $Q$ we have
$\partial_a W\in K Q$. In this case we may consider $\cJ'(W)$, the
ideal of the (non complete) path algebra $K Q$ which is generated by the
cyclic derivatives $\partial_a W$ for all arrows $a$ of $Q$. Clearly,
$\cJ'(W)$ is contained in the ideal of $K Q$ which is generated by 
the arrows. Suppose  moreover  that for some $m\in\mathbb{N}$  
each path  in $Q$ of length greater than $m$ belongs to $\cJ'(W)$. 
Then it is not hard to see that the (finite dimensional) algebra
$K Q/\cJ'(W)$, which we call \emph{fake Jacobian algebra}, 
is isomorphic to the Jacobian algebra $\cP(Q,W)$. This will
occur for all the Jacobian algebras which are studied explicitly in
this paper.
\end{remark}

\subsection{Cluster tilted algebras of tubular type as Jacobian algebras}
\label{ssec:TubJac}
Let $\XX$ be a weighted projective line of tubular type and 
$\pi\colon\coh(\XX)\rightarrow\cC_\XX$ the canonical projection. It is
easy to see that for each (basic) cluster tilting object 
$T=\oplus_{i=1}^n T_i\in\cC_\XX$ there exists up to isomorphism a unique
(basic) tilting sheaf $\tilde{T}=\oplus_{i=1}^n\tilde{T}_i\in\coh(\XX)$ with 
$\pi(\tilde{T})\cong T$.  In this situation, let us write the basic tubular
algebra $A:=\End_{\coh(\XX)}(\tilde{T})^\opp$ as a path algebra modulo an
admissible ideal, say 
$A\cong K Q_{\tilde{T}}/I_{\tilde{T}}$.
We may assume that the vertex $i$ of $Q_{\tilde{T}}$ corresponds to the
summand $\tilde{T}_i$ of $\tilde{T}$, and that $I$ is generated by a minimal
set of readable relations $\rho_1,\ldots,\rho_m$. In particular, each $\rho_i$
is a linear combination of paths from $s_i$ to $t_i$ in $Q_{\tilde{T}}$.
Let $Q_T$ be the quiver obtained from $Q_{\tilde{T}}$ by inserting a new arrow
$a_i\colon t_i\rightarrow s_i$ for $i=1,\ldots,m$.
Then $W_T:=\sum_{i=1}^r\rho_i a_i$ is a potential for $Q_T$, and we observe that
$Q_T$ has a no loops or oriented 2-cycles since $Q_{\tilde{T}}$ has no oriented
cycles.

\begin{remark} The QPs described in Figure~\ref{f:0} are obtained by the
just mentioned construction. In each case, if we delete the arrows which 
point from left to right, we obtain the quiver $\tilde{Q}$ of a tubular 
algebra $A$. If we write $A=K \tilde{Q}/I$ then $I$ is generated by
relations $\rho_1,\ldots,\rho_m$, one relation for each deleted arrow.
\end{remark} 

\begin{proposition} \label{prp:tubjac}
Let $\XX$ be a weighted projective line of tubular type and $T\in\cC_\XX$
a basic cluster tilting object. Then we have:
\begin{itemize}
\item[(a)]
$\End_{\cC_\XX}(T)^\opp\cong \cP(Q_T, W_T)$ for the above described
quiver with potential associated to $T$.
\item[(b)] Let for $k=1,\ldots, n$ be $T'_k\in\cC_\XX$ be the unique
object such that $\mu_k(T):= T/T_k\oplus T'_k$ is a basic cluster tilting 
object not isomorphic to $T$. Then the QP-mutation $\mu_k(Q_T,W_T)$,
as defined in~\cite[Sec.~5]{DWZ1}, is right equivalent to 
$(Q_{\mu_k(T)}, W_{\mu_k(T)})$.
In particular, the above mentioned QPs are non-degenerate.
\item[(c)]
$(Q_T,W_T)$ can be obtained from one of the QPs described in Figure~\ref{f:0}
by a finite sequence of QP-mutations.
\end{itemize}
\end{proposition}

\begin{proof}
(a) $A$ is of global dimension 2, moreover the derived categories
$D^b(\Lmod{A})$ and $D^b(\coh(\XX))$ are triangle equivalent. Since
$\coh(\XX)$ is a hereditary category, the cluster category $\cC_A$
can be identified naturally with $\cC_\XX$~\cite[Thm.~1]{Ke}. Now, our
claim follows from~\cite[Thm.~6.12a)]{Kel11-defoCY}.

For the proof of (b) we have to distinguish two cases.\\[.5ex]
(b1) $\XX$ is of type $(2,2,2,2;\lam)$. 
There are only relatively few families (which depend on the parameter
$\lam\in K\setminus\{0,1\}$) of tubular algebras of type 
$(2,2,2,2;\lambda)$. They were listed for example in~\cite[3.2]{Sko}.
It is straightforward to check that they yield the four quivers with
potential described in~\cite[p.117]{GeKeOp}. It is an easy exercise
to see that the corresponding Jacobian algebras 
(for $\lambda\in K\setminus\{0,1\}$) are finite-dimensional and
(weakly)  symmetric. In particular, for each cluster tilting object
$T\in\cC_\XX$ the endomorphism ring $\End_{\cC_\XX}(T)^\opp$ fulfills
trivially the vanishing condition ($\mathrm{C}_k$) from~\cite[p.~855]{BIRS}
for each vertex $k$ of its quiver. Thus the claim follows 
from~\cite[Thm.~5.2]{BIRS}.\\[.5ex]
(b2) $\XX$ is of type  $(3,3,3), (4,4,2)$ or $(6,3,2)$. In this case, it
was pointed out in~\cite[Sec.~3.2]{BACH} that $\cC_\XX$ is triangle equivalent
to a stable category $\underline{\cC}_M$ for a certain Frobenius category 
$\cC_M$, which is a subcategory of the module category a 
preprojective algebra of Dynkin type. See~\cite[Sec.~3.1]{BACH} for more
details on the definition of $\cC_M$. (In the cases $(3,3,3)$
resp. $(6,3,2)$ one takes for $\cC_M$ the full category of finite dimensional
modules over a preprojective algebra for a quiver $Q'$ of type $\mathsf{D}_4$
resp. $\mathsf{A}_5$. In case $(4,4,2)$ the category $\cC_M$ is a
proper subcategory of the module category of a quiver $Q'$ of type
$\mathsf{A}_7$. However, we don't need the specific description here.)
The categories $\cC_M$ are up to duality special cases of the categories
$\cC_w$ discussed in~\cite[Sec.~6]{BIRS}. More precisely, $w$ is the
element of the Coxeter group associated to $Q'$ which sends precisely
those positive roots for $Q'$, which are of the form $\ddim M'$ for an
indecomposable direct summand $M'$ of $M$, to a negative root. For example,
in the cases $(3,3,3)$ and $(6,3,2)$ the element $w$ would be the
longest element of the Coxeter group. Thus, by combining~(a),    
\cite[Cor.~5.4]{BIRS} and~\cite[Thm.~6.6]{BIRS} we see that the claim
holds for all cluster tilting objects $T$ which are in the mutation class
of one of the canonical cluster tilting objects associated to a
reduced expression for $w$. Finally, for a tubular cluster category
the exchange graph of cluster tilting objects is 
connected~\cite[Thm.~8.8]{BaKuLe}, so that our claim holds for all cluster 
tilting objects.

(c) This is clear by (a) and~(b) together with the just mentioned
connectedness of the exchange graph of cluster tilting objects for
tubular cluster categories.
\end{proof}

\begin{remark} \label{rem:tubjac}
(1) Proposition~\ref{prp:tubjac} provides for tubular
Jacobian algebras a description similar the the well known description
of cluster tilted algebras of acyclic type. It would be nice to have a 
more elementary proof for this result.

(2) Let $\XX$ and $\XX'$ be weighted projective lines of type
$(2,2,2,2; \lam)$ resp. $(2,2,2,2;\lam')$. It is well known that
$\coh(\XX)$ and $\coh(\XX')$ are equivalent abelian categories
iff $D^b(\coh(\XX))$ and  $D^b(\coh(\XX'))$ are triangle equivalent 
iff $\cC_{\XX}$ and $\cC_{\XX'}$ are triangle equivalent 
iff $\lam'\in O(\lam):=\{\lam,\lam^{-1}, 1-\lam, 1-\lam^{-1}, (1-\lam)^{-1}, 
\frac{\lam}{\lam -1}\}$, see also~\cite[p.~157]{HaRi}.
This fact is reflected by the following
observation:  two QPs $(Q^{(2)}, W^{(2)}_\lam)$ and $(Q^{(2)}, W^{(2)}_{\lam'})$ 
(from~Figure~\ref{f:0}) are in the same (per-) mutation class of QPs 
iff $\lam'\in O(\lam)$.  However, the corresponding Jacobian algebras are only 
isomorphic if $\lam'\in\{\lam,\lam^{-1}\}$. See also~\cite[9.9]{GeLaSc}
for more results of this kind. The cumbersome calculations will
appear elsewhere. 
\end{remark}

\section{Galois Coverings of tubular Jacobian algebras}\label{sec:Gal} 
\subsection{General setup}\label{ssec:GalGen} 
Set from now on $W^{(2)}:= W^{(2)}_\lam$ for a fixed $\lam\in K\setminus\{0,1\}$.
In this section we write down for each of the  the Jacobian algebras 
$\cP(Q^{(i)},W^{(i)})$ with $i=1,\ldots,4$ (see Figure~\ref{f:0}) a 
Galois covering.
We proceed as follows: We first verify that the Jacobian ideal $\cJ'(W^{(i)})$,
see Remark~\ref{rem:fjac}, of the path algebra $K Q^{(i)}$ is admissible.
As a byproduct we obtain the dimension vectors of the indecomposable 
projectives. We display them according to the shape of the quiver, 
and highlight the entry which corresponds to the top of the corresponding 
projective.

Next, if we assign to the arrows of $Q^{(i)}$ which point from left to right
degree $1$, and degree $0$ to the remaining arrows, $K Q^{(i)}$ becomes
naturally $\ZZ$-graded and $W^{(i)}\in K Q^{(i)}$  is homogeneous of degree $1$.
Thus the cyclic derivatives $\partial_a W^{(i)}$ for $a\in Q^{(i)}_1$ are 
also homogeneous. We conclude that $K Q^{(i)}/\cJ'(W^{(i)})=\cP(Q^{(i)},W^{(i)})$ is
$\ZZ$-graded. It is now straightforward to write down the Galois covering
mentioned in Remark~\ref{rem:gal} as a quiver with relations:
$\tilde{Q}^{(i)}_0= Q^{(i)}_0\times\ZZ$,  $\tilde{Q}^{(i)}_1= Q_1\times\ZZ$
with $s(a,z)= (s(a),z)$ and $t(a,z)= (t(a), z+\operatorname{deg}(a))$
for all $a\in Q^{(i)}_1$ and $z\in\ZZ$. Thanks to our observation that
$K Q^{(i)}/\cJ'(W)=\cP(Q^{(i)},W^{(i)})$ the defining relations for the
covering are just the obvious ``lifts'' of the cyclic derivatives
$\partial_a W^{(i)}$  with $a\in Q^{(i)}_1$. 
Given the easy structure of the dimension vectors of the indecomposable 
preprojective modules over the Jacobian algebra, it is easy to derive
the dimension vectors of the indecomposable projectives of the
Galois covering. This will be useful in the next section where we 
identify those Galois coverings as iterated tubular algebras.

\subsection{Case $(3,3,3)$} \label{cov:333}
We consider the quiver with potential $(Q^{(1)}, W^{(1)})$ from Figure~\ref{f:0}.
The  dimension vectors of its indecomposable projective modules over the
corresponding (fake) Jacobian algebra are as follows: 
\[
    \begin{array}{c}
      \phantom{000}0\phantom{00}0  \\
      \phantom{000}0\phantom{00}0  \\
              {\bf{1}}\phantom{00}1 \\
              \phantom{000}0 \phantom{00}0
    \end{array},\quad
    \begin{array}{c}
      \phantom{000}{\bf{1}}\phantom{00}1\\
      \phantom{000}0\phantom{00}0\\
      1\phantom{00}1 \\
      \phantom{000}0 \phantom{00}0
    \end{array},\quad
    \begin{array}{c}
      \phantom{000}0\phantom{00}0\\
      \phantom{000}{\bf{1}}\phantom{00}1\\
      1\phantom{00}1 \\
      \phantom{000}0 \phantom{00}0
    \end{array},\quad
    \begin{array}{c}
      \phantom{000}0\phantom{00}0\\
      \phantom{000}0\phantom{00}0\\
      1\phantom{00}1 \\
      \phantom{000}{\bf{1}}\phantom{00}1
    \end{array},\quad
    \begin{array}{c}
      \phantom{000}1\phantom{00}0\\
      \phantom{000}1\phantom{00}0\\
      2\phantom{00}{\bf{1}} \\
      \phantom{000}1 \phantom{00}0
    \end{array},
    \]

    \[
    \begin{array}{c}
      \phantom{000}0\phantom{00}{\bf{1}}\\
      \phantom{000}1\phantom{00}0\\
      1\phantom{00}1 \\
      \phantom{000}1 \phantom{00}0
    \end{array},\quad
    \begin{array}{c}
      \phantom{000}1\phantom{00}0\\
      \phantom{000}0\phantom{00}{\bf{1}}\\
      1\phantom{00}1 \\
      \phantom{000}1 \phantom{00}0
    \end{array}\text{\qquad and}\quad
    \begin{array}{c}
      \phantom{000}1\phantom{00}0\\
      \phantom{000}1\phantom{00}0\\
      1\phantom{00}1 \\
      \phantom{000}0 \phantom{00}{\bf{1}}
    \end{array}
    \]   

From the considerations in~\ref{ssec:GalGen} we obtain the following
Galois covering:

    \begin{center}
      \setlength{\unitlength}{0.7mm}
      \begin{picture}(120,50)
        \put(60,25){
          \put(-45,15){\circle*{2}} 
          \put(-65,15){\circle*{2}} 
          \put(-45,5){\circle*{2}}
          \put(-65,5){\circle*{2}}
          \put(-45,-15){\circle*{2}}
          \put(-65,-15){\circle*{2}}
          \put(-55,-5){\circle*{2}}
          \put(-75,-5){\circle*{2}}
          \multiput(-46,14)(-20,0){2}{\vector(-1,-2){8.7}}
          \multiput(-46,4)(-20,0){2}{\vector(-1,-1){8}}
          \multiput(-46,-14)(-10,10){2}{\vector(-1,1){8}}
          \multiput(-56,-6)(-10,10){1}{\vector(-1,-1){8}}
          \multiput(-66,-14)(-20,0){1}{\vector(-1,1){8}}
          \multiput(-55.7,-4)(-20,0){1}{\vector(-1,2){8.8}}
          \qbezier[12](-63.5,15)(-55,15)(-46.5,15)
          \qbezier[12](-63.5,5)(-55,5)(-46.5,5)
          \qbezier[12](-63.5,-15)(-55,-15)(-46.5,-15)
          \qbezier[12](-73.5,-5)(-65,-5)(-56.5,-5)

          \put(-62.2,10){$b$}
          \put(-50.5,10){$c$}
          \put(-63,-2.5){$e$}
          \put(-49.6,-2.5){$f$}
          \put(-49.4,-10.5){$h$}
          \put(-62.4,-10.5){$i$}
          \put(-76.5,0){$k$}
          \put(-68.8,-2.5){$l$}            
          \put(-74.2,-12.4){$m$}            

          \put(15,15){\circle*{2}} 
          \put(-5,15){\circle*{2}} 
          \put(15,5){\circle*{2}}
          \put(-5,5){\circle*{2}}
          \put(15,-15){\circle*{2}}
          \put(-5,-15){\circle*{2}}
          \put(5,-5){\circle*{2}}
          \put(-15,-5){\circle*{2}}
          \multiput(14,14)(-20,0){2}{\vector(-1,-2){8.7}}
          \multiput(14,4)(-20,0){2}{\vector(-1,-1){8}}
          \multiput(14,-14)(-10,10){2}{\vector(-1,1){8}}
          \multiput(4,-6)(-10,10){1}{\vector(-1,-1){8}}
          \multiput(-6,-14)(-20,0){1}{\vector(-1,1){8}}
          \multiput(4.3,-4)(-20,0){1}{\vector(-1,2){8.8}}
          \multiput(-7,-15)(0,20){2}{\vector(-1,0){36}}
          \multiput(-7,15)(0,0){2}{\vector(-1,0){36}}
          \multiput(-17,-5)(0,0){2}{\vector(-1,0){36}}  
          \qbezier[12](-3.5,15)(5,15)(13.5,15)
          \qbezier[12](-3.5,5)(5,5)(13.5,5)
          \qbezier[12](-3.5,-15)(5,-15)(13.5,-15)
          \qbezier[12](-13.5,-5)(-5,-5)(3.5,-5)

          \put(-26,16){$a$}
          \put(-2.2,10){$b$}
          \put(9.5,10){$c$}
          \put(-26,6){$d$}
          \put(-3,-2.5){$e$}
          \put(10.4,-2.5){$f$}
          \put(-26,-13.3){$g$}
          \put(10.6,-10.5){$h$}
          \put(-2.4,-10.5){$i$}
          \put(-32,-3){$j$}
          \put(-16.5,0){$k$}
          \put(-8.8,-2.5){$l$}            
          \put(-14.2,-12.4){$m$}

          \put(75,15){\circle*{2}} 
          \put(55,15){\circle*{2}}
          \put(75,5){\circle*{2}}
          \put(55,5){\circle*{2}}
          \put(75,-15){\circle*{2}}
          \put(55,-15){\circle*{2}}
          \put(65,-5){\circle*{2}}
          \put(45,-5){\circle*{2}}
          \multiput(74,14)(-20,0){2}{\vector(-1,-2){8.7}}
          \multiput(74,4)(-20,0){2}{\vector(-1,-1){8}}
          \multiput(74,-14)(-10,10){2}{\vector(-1,1){8}}
          \multiput(64,-6)(-10,10){1}{\vector(-1,-1){8}}
          \multiput(54,-14)(-20,0){1}{\vector(-1,1){8}}
          \multiput(64.3,-4)(-20,0){1}{\vector(-1,2){8.8}}
          \multiput(53,-15)(0,20){2}{\vector(-1,0){36}}
          \multiput(53,15)(0,0){2}{\vector(-1,0){36}}
          \multiput(43,-5)(0,0){2}{\vector(-1,0){36}}  
          \qbezier[12](56.5,15)(65,15)(73.5,15)
          \qbezier[12](56.5,5)(65,5)(73.5,5)
          \qbezier[12](56.5,-15)(65,-15)(73.5,-15)
          \qbezier[12](46.5,-5)(55,-5)(63.5,-5)

          \put(34,16){$a$}
          \put(57.8,10){$b$}
          \put(69.5,10){$c$}
          \put(34,6){$d$}
          \put(57,-2.5){$e$}
          \put(70.4,-2.5){$f$}
          \put(34,-13.3){$g$}
          \put(70.6,-10.5){$h$}
          \put(57.6,-10.5){$i$}
          \put(28,-3){$j$}
          \put(43.5,0){$k$}
          \put(51.2,-2.5){$l$}            
          \put(45.8,-12.4){$m$}            

          \put(-73,15){$\ldots$}
          \put(-73,5){$\ldots$}
          \put(-73,-15){$\ldots$}
          \put(-83,-5){$\ldots$}

          \put(77,15){$\ldots$}
          \put(77,5){$\ldots$}
          \put(77,-15){$\ldots$}
          \put(66.5,-5){$\ldots$}
        }     
      \end{picture}
    \end{center}

In this diagram the dotted lines indicate zero relations or  
commutativity relations. In addition to these relations, all the squares 
commute. 

\subsection{Case $(2,2,2,2;\lambda)$} \label{cov:2222}
We consider the quiver with potential $(Q^{(2)}, W^{(2)}_\lam)$ from 
Figure~\ref{f:0}.
An easy calculation shows:  every path of length $4$ belongs 
to $\cJ(W^{(2)}_\lam)$; the same is true for paths of length $3$ which are 
not cycles.

The dimension vectors of the indecomposable projective modules 
over the corresponding (fake) Jacobian algebra are as follows: 
\[
    \begin{array}{c}
      \phantom{00}{\bf{2}} \phantom{00}1 \phantom{00}1\\
      \phantom{00}0 \phantom{00}1 \phantom{00}1     
    \end{array},\quad
    \begin{array}{c}
      \phantom{00}0 \phantom{00}1 \phantom{00}1\\
      \phantom{00}{\bf{2}} \phantom{00}1 \phantom{00}1     
    \end{array},\quad
    \begin{array}{c}
      \phantom{00}1 \phantom{00}{\bf{2}} \phantom{00}1\\
      \phantom{00}1 \phantom{00}0 \phantom{00}1     
    \end{array},\quad
    \begin{array}{c}
      \phantom{00}1 \phantom{00}0 \phantom{00}1\\
      \phantom{00}1 \phantom{00}{\bf{2}} \phantom{00}1     
    \end{array},\quad
\]

\[
    \begin{array}{c}
      \phantom{00}1 \phantom{00}1 \phantom{00}{\bf{2}}\\
      \phantom{00}1 \phantom{00}1 \phantom{00}0     
    \end{array}   \text{\qquad and}\quad
    \begin{array}{c}
      \phantom{00}1 \phantom{00}1 \phantom{00}0\\
      \phantom{00}1 \phantom{00}1 \phantom{00}{\bf{2}}     
    \end{array},\quad   
 \]   

From the considerations in~\ref{ssec:GalGen} we obtain the following
Galois covering:

    \begin{center} 
      \setlength{\unitlength}{0.7mm}
      \begin{picture}(80,50)
        \put(40,25){
          \put(-40,10){\circle*{2}} 
          \put(-40,-10){\circle*{2}}
          \put(-60,10){\circle*{2}}
          \put(-60,-10){\circle*{2}}
          \put(-80,10){\circle*{2}}
          \put(-80,-10){\circle*{2}}
          \multiput(-42,10)(-20,0){2}{\vector(-1,0){16}}
          \multiput(-42,-10)(-20,0){2}{\vector(-1,0){16}}
          \multiput(-41.5,8.5)(-20,0){2}{\vector(-1,-1){17}}
          \multiput(-41.5,-8.5)(-20,0){2}{\vector(-1,1){17}}
          \put(-71,10.3){$a$}
          \put(-51,10.3){$b$}
          \put(-31,10.3){$c$}
          \put(-75,5){$d$}
          \put(-75,-2.8){$e$}
          \put(-54,5){$f$}
          \put(-54.5,-1.8){$g$}
          \put(-35,5){$h$}
          \put(-35,-2.8){$i$}
          \put(-71,-9.7){$j$}
          \put(-51,-9.7){$k$}
          \put(-31,-9.7){$l$}
          \multiput(-22,10)(0,-20){2}{\vector(-1,0){16}}
          \multiput(-21.5,8.5)(-20,0){1}{\vector(-1,-1){17}}
          \multiput(-21.5,-8.5)(-20,0){1}{\vector(-1,1){17}}

          \put(20,10){\circle*{2}} 
          \put(20,-10){\circle*{2}}
          \put(0,10){\circle*{2}}
          \put(0,-10){\circle*{2}}
          \put(-20,10){\circle*{2}}
          \put(-20,-10){\circle*{2}}
          \multiput(18,10)(-20,0){2}{\vector(-1,0){16}}
          \multiput(18,-10)(-20,0){2}{\vector(-1,0){16}}
          \multiput(18.5,8.5)(-20,0){2}{\vector(-1,-1){17}}
          \multiput(18.5,-8.5)(-20,0){2}{\vector(-1,1){17}}
          \put(-11,10.3){$a$}
          \put(9,10.3){$b$}
          \put(29,10.3){$c$}
          \put(-15,5){$d$}
          \put(-15,-2.8){$e$}
          \put(6,5){$f$}
          \put(5.5,-1.8){$g$}
          \put(25,5){$h$}
          \put(25,-2.8){$i$}
          \put(-11,-9.7){$j$}
          \put(9,-9.7){$k$}
          \put(29,-9.7){$l$}
          \multiput(38,10)(0,-20){2}{\vector(-1,0){16}}
          \multiput(38.5,8.5)(-20,0){1}{\vector(-1,-1){17}}
          \multiput(38.5,-8.5)(-20,0){1}{\vector(-1,1){17}}
          
          \put(80,10){\circle*{2}} 
          \put(80,-10){\circle*{2}}
          \put(60,10){\circle*{2}}
          \put(60,-10){\circle*{2}}
          \put(40,10){\circle*{2}}
          \put(40,-10){\circle*{2}}
          \multiput(78,10)(-20,0){2}{\vector(-1,0){16}}
          \multiput(78,-10)(-20,0){2}{\vector(-1,0){16}}
          \multiput(78.5,8.5)(-20,0){2}{\vector(-1,-1){17}}
          \multiput(78.5,-8.5)(-20,0){2}{\vector(-1,1){17}}
          \put(49,10.3){$a$}
          \put(69,10.3){$b$}
          \put(45,5){$d$}
          \put(45,-2.8){$e$}
          \put(66,5){$f$}
          \put(65.5,-1.8){$g$}
          \put(49,-9.7){$j$}
          \put(69,-9.7){$k$}          

          \put(-88,10){$\dots$}
          \put(-88,-10){$\dots$}
          \put(82,10){$\dots$}
          \put(82,-10){$\dots$}
        }     
      \end{picture}
    \end{center}
In which all the squares commutes up to possibly a factor $\lambda$.

\subsection{Case $(4,4,2)$} \label{cov:442}.
We consider the quiver with potential $(Q^{(3)}, W^{(3)})$ from Figure~\ref{f:0}:
The  dimension vectors of its indecomposable projective modules over the
corresponding (fake) Jacobian algebra are as follows:  
    \[
    \begin{array}{c}
      1  \\
      1\phantom{00}1  \\
      1\phantom{00}1\phantom{00}{\bf{1}}\\
      1 \phantom{00}1\\
      1
    \end{array},\quad
    \begin{array}{c}
      1  \\
      1\phantom{00}{\bf{1}}\\
      1\phantom{00}1\phantom{00}0 \\
      1 \phantom{00}0\\
      0
    \end{array},\quad
    \begin{array}{c}
      0\\
      1\phantom{00}0\\
      1\phantom{00}1\phantom{00}0\\
      1\phantom{00}{\bf{1}}\\
      1
    \end{array},\quad
    \begin{array}{c}
      {\bf{1}}\\
      1\phantom{00}0\\
      1\phantom{00}0\phantom{00}0\\
      0 \phantom{00}0\\
      0
    \end{array},\quad
    \begin{array}{c}
      0  \\
      1\phantom{00}1  \\
      1\phantom{00}{\bf{2}}\phantom{00}1\\
      1\phantom{00}1\\
      0
    \end{array},
    \]

    \[
    \begin{array}{c}
      0  \\
      0\phantom{00}0  \\
      1\phantom{00}0\phantom{00}0 \\
      1 \phantom{00}0\\
      {\bf{1}}
    \end{array},\quad 
    \begin{array}{c}
      0\\
      {\bf{1}}\phantom{00}1\\
      1\phantom{00}1\phantom{00}0\\
      0\phantom{00}0\\
      0
    \end{array},\quad
    \begin{array}{c}
      0\\
      0\phantom{00}0\\
      1\phantom{00}1\phantom{00}0\\
      {\bf{1}}\phantom{00}1\\
      0
    \end{array}\text{\qquad and}\quad
    \begin{array}{c}
      0\\
      0\phantom{00}0\\
      {\bf{1}}\phantom{00}1\phantom{00}1\\
      0\phantom{00}0\\
      0
    \end{array}.
    \]   

From the considerations in~\ref{ssec:GalGen} we obtain the following
Galois covering:

    \begin{center}
      \setlength{\unitlength}{0.7mm}
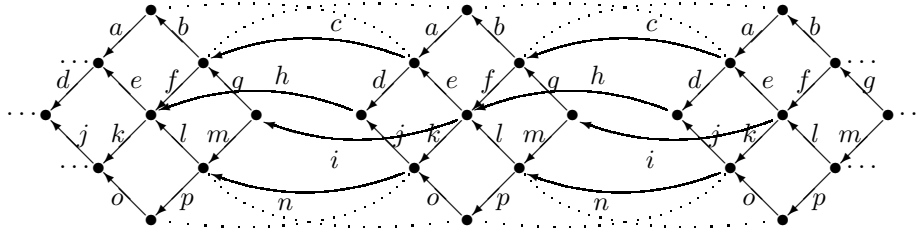
      \begin{picture}(120,60)
        \put(60,30){
          \put(-60,0){\circle*{2}} 
          \put(-80,0){\circle*{2}}
          \put(-40,0){\circle*{2}}
          \put(-50,10){\circle*{2}}
          \put(-70,10){\circle*{2}}
          \put(-60,20){\circle*{2}}
          \put(-50,-10){\circle*{2}}
          \put(-70,-10){\circle*{2}}
          \put(-60,-20){\circle*{2}}
          \multiput(-61,1)(10,10){2}{\vector(-1,1){8}}
          \multiput(-51,-9)(10,10){2}{\vector(-1,1){8}}
          \multiput(-71,-9)(10,-10){2}{\vector(-1,1){8}}
          \multiput(-71,9)(10,10){2}{\vector(-1,-1){8}}
          \multiput(-61,-1)(10,10){2}{\vector(-1,-1){8}} 
          \multiput(-51,-11)(10,10){2}{\vector(-1,-1){8}}
          \put(-68,15.2){$a$}
          \put(-55,15.2){$b$}
          \put(-78,5.2){$d$}
          \put(-64,5.2){$e$}
          \put(-57.5,5.2){$f$}
          \put(-44.9,5.2){$g$}
          \put(-74,-5){$j$}
          \put(-67.7,-5){$k$}
          \put(-54.6,-5){$l$}
          \put(-49.4,-5){$m$}
          \put(-67.7,-17){$o$}
          \put(-54.4,-17){$p$}

          \qbezier[22](-50,10)(-40,25)(0,20) 
          \qbezier[22](-60,20)(-20,25)(-10,10)
          \qbezier[22](-60,-20)(-20,-25)(-10,-10)
          \qbezier[22](-50,-10)(-40,-25)(0,-20)

          \put(0,0){\circle*{2}} 
          \put(-20,0){\circle*{2}}
          \put(20,0){\circle*{2}}
          \put(10,10){\circle*{2}}
          \put(-10,10){\circle*{2}}
          \put(0,20){\circle*{2}}
          \put(10,-10){\circle*{2}}
          \put(-10,-10){\circle*{2}}
          \put(0,-20){\circle*{2}}
          \multiput(-1,1)(10,10){2}{\vector(-1,1){8}}
          \multiput(9,-9)(10,10){2}{\vector(-1,1){8}}
          \multiput(-11,-9)(10,-10){2}{\vector(-1,1){8}}
          \multiput(-11,9)(10,10){2}{\vector(-1,-1){8}}
          \multiput(-1,-1)(10,10){2}{\vector(-1,-1){8}} 
          \multiput(9,-11)(10,10){2}{\vector(-1,-1){8}}
          \put(-8,15.2){$a$}
          \put(5,15.2){$b$}
          \put(-18,5.2){$d$}
          \put(-4,5.2){$e$}
          \put(2.5,5.2){$f$}
          \put(15.1,5.2){$g$}
          \put(-14,-5){$j$}
          \put(-7.7,-5){$k$}
          \put(5.4,-5){$l$}
          \put(10.6,-5){$m$}
          \put(-7.7,-17){$o$}
          \put(5.6,-17){$p$}

          \qbezier[22](10,10)(20,25)(60,20) 
          \qbezier[22](0,20)(40,25)(50,10)
          \qbezier[22](0,-20)(40,-25)(50,-10)         
          \qbezier[22](10,-10)(20,-25)(60,-20)

          \put(60,0){\circle*{2}} 
          \put(40,0){\circle*{2}}
          \put(80,0){\circle*{2}}
          \put(70,10){\circle*{2}}
          \put(50,10){\circle*{2}}
          \put(60,20){\circle*{2}}
          \put(70,-10){\circle*{2}}
          \put(50,-10){\circle*{2}}
          \put(60,-20){\circle*{2}}
          \multiput(59,1)(10,10){2}{\vector(-1,1){8}}
          \multiput(69,-9)(10,10){2}{\vector(-1,1){8}}
          \multiput(49,-9)(10,-10){2}{\vector(-1,1){8}}
          \multiput(49,9)(10,10){2}{\vector(-1,-1){8}}
          \multiput(59,-1)(10,10){2}{\vector(-1,-1){8}} 
          \multiput(69,-11)(10,10){2}{\vector(-1,-1){8}}
          \put(52,15.2){$a$}
          \put(65,15.2){$b$}
          \put(42,5.2){$d$}
          \put(56,5.2){$e$}
          \put(62.5,5.2){$f$}
          \put(75.1,5.2){$g$}
          \put(46,-5){$j$}
          \put(52.3,-5){$k$}
          \put(65.4,-5){$l$}
          \put(70.6,-5){$m$}
          \put(52.3,-17){$o$}
          \put(65.6,-17){$p$}

          \qbezier(48,11)(30,18)(13,11.4) 
          \qbezier(38,1)(20,8)(3,1.4)
          \qbezier(58,-1)(40,-8)(23,-1.8)
          \qbezier(48,-11)(30,-18)(13,-11.8)
          \put(17.5,12.9){\vector(-3,-1){6}}
          \put(17.5,-13.3){\vector(-3,1){6}}
          \put(7.5,2.9){\vector(-3,-1){6}}
          \put(27.5,-3.3){\vector(-3,1){6}}
          \put(-26,16){$c$} 
          \put(-36.6,6){$h$}
          \put(-26,-10){$i$}
          \put(-36,-18){$n$}

          \qbezier(-12,11)(-30,18)(-47,11.4) 
          \qbezier(-22,1)(-40,8)(-57,1.4)
          \qbezier(-2,-1)(-20,-8)(-37,-1.8)
          \qbezier(-12,-11)(-30,-18)(-47,-11.8)
          \put(-42.5,12.9){\vector(-3,-1){6}}
          \put(-42.5,-13.3){\vector(-3,1){6}}
          \put(-52.5,2.9){\vector(-3,-1){6}}
          \put(-32.5,-3.3){\vector(-3,1){6}}
          \put(34,16){$c$}
          \put(23.4,6){$h$}
          \put(34,-10){$i$}
          \put(24,-18){$n$}

          \put(82,0){$\dots$}
          \put(72,10){$\dots$}
          \put(72,-10){$\dots$}
          \put(-87.5,0){$\dots$}
          \put(-77.5,10){$\dots$}
          \put(-77.5,-10){$\dots$}
        }     
      \end{picture}
    \end{center}
In this diagram the dotted lines indicate zero relations. 
In addition to these  relations, all the squares commute.

\subsection{Case$(6,3,2)$} \label{cov:632}
We consider the quiver with potential $(Q^{(3)}, W^{(3)})$ from Figure~\ref{f:0}.
The corresponding Jacobian algebra $\cP(Q^{(4)},W^{(4)})$ is finite dimensional, 
the dimension vectors of its indecomposable projective modules are as follows: 

    \[
    \begin{array}{l}
      \phantom{00}0\phantom{00}{\bf{1}}\\
      0\phantom{0.0}1  \\
      \phantom{00}1\phantom{00}0\\
      1\phantom{0.0}0\\
      \phantom{00}0\phantom{00}0
    \end{array},\quad
    \begin{array}{l}
      \phantom{00}1\phantom{00}0\\
      1\phantom{0.0}1\\
      \phantom{00}1\phantom{00}{\bf{1}}\\
      1\phantom{0.0}1\\
      \phantom{00}1\phantom{00}0
    \end{array},\quad
    \begin{array}{l}
      \phantom{00}0\phantom{00}0\\
      1\phantom{0.0}0\\
      \phantom{00}1\phantom{00}0\\
      0\phantom{0.0}1\\
      \phantom{00}0\phantom{00}{\bf{1}}
    \end{array},\quad
    \begin{array}{l}
      \phantom{00}1\phantom{00}0\\
      1\phantom{0.0}{\bf{1}}\\
      \phantom{00}1\phantom{00}0\\
      1\phantom{0.0}0\\
      \phantom{00}0\phantom{00}0
    \end{array},\quad
    \begin{array}{l}
      \phantom{00}0\phantom{00}0\\
      1\phantom{0.0}0\\
      \phantom{00}1\phantom{00}0\\
      1\phantom{0.0}{\bf{1}}\\
      \phantom{00}1\phantom{00}0
    \end{array},
    \]

    \[
    \begin{array}{l}
      \phantom{00}{\bf{1}}\phantom{00}1  \\
      1\phantom{00.}1  \\
      \phantom{00}0\phantom{00}0\\
      0\phantom{00.}0\\
      \phantom{00}0\phantom{00}0
    \end{array},\quad
    \begin{array}{l}
      \phantom{00}0\phantom{00}0\\
      1\phantom{00.}1\\
      \phantom{00}{\bf{1}}\phantom{00}1\\
      1\phantom{00.}1\\
      \phantom{00}0\phantom{00}0
    \end{array},\quad
    \begin{array}{l}
      \phantom{00}0\phantom{00}0\\
      0\phantom{00.}0\\
      \phantom{00}0\phantom{00}0\\
      1\phantom{00.}1\\
      \phantom{00}{\bf{1}}\phantom{00}1
    \end{array},\quad
    \begin{array}{l}
      \phantom{00}0\phantom{00}0\\
              {\bf{1}}\phantom{00.}1\\
              \phantom{00}0\phantom{00}0\\
              0\phantom{00.}0\\
              \phantom{00}0\phantom{00}0
    \end{array} \text{\quad and} \quad
    \begin{array}{l}
      \phantom{00}0\phantom{00}0\\
      0\phantom{00.}0\\
      \phantom{00}0\phantom{00}0\\
              {\bf{1}}\phantom{00.}1\\
              \phantom{00}0\phantom{00}0
    \end{array},
    \]

From the considerations in~\ref{ssec:GalGen} we obtain the following
Galois covering:

    \begin{center} 
      \setlength{\unitlength}{0.7mm}
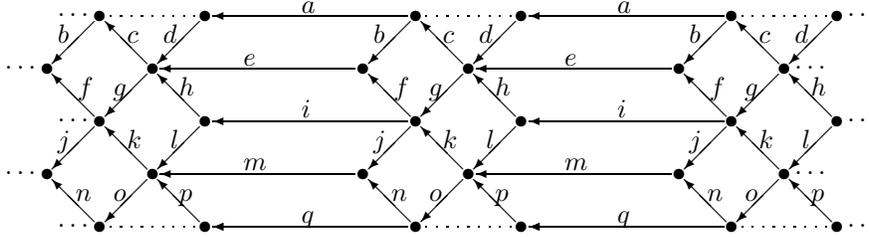
      \begin{picture}(120,60)
        \put(60,30){
          \put(-65,0){\circle*{2}}
          \put(-45,20){\circle*{2}}
          \put(-45,0){\circle*{2}}
          \put(-55,10){\circle*{2}}
          \put(-75,10){\circle*{2}}
          \put(-65,20){\circle*{2}}
          \put(-55,-10){\circle*{2}}
          \put(-75,-10){\circle*{2}}
          \put(-65,-20){\circle*{2}}
          \put(-45,-20){\circle*{2}}
          \multiput(-66,-1)(0,20){2}{\vector(-1,-1){8}}
          \multiput(-56,-11)(0,20){2}{\vector(-1,-1){8}}
          \multiput(-46,-1)(0,20){2}{\vector(-1,-1){8}}
          \multiput(-66,1)(0,-20){2}{\vector(-1,1){8}}
          \multiput(-56,-9)(0,20){2}{\vector(-1,1){8}} 
          \multiput(-46,1)(0,-20){2}{\vector(-1,1){8}}
          \put(-73,15){$b$}
          \put(-59.8,15){$c$}
          \put(-53,15){$d$}
          \put(-69.2,5){$f$}
          \put(-62.5,5){$g$}
          \put(-49.9,5){$h$}
          \put(-73,-5){$j$}
          \put(-59.8,-5){$k$}
          \put(-51.6,-5){$l$}
          \put(-69.6,-15){$n$}
          \put(-62.4,-15){$o$}
          \put(-49.9,-15){$p$}
          \qbezier[10](-45,20)(-55,20)(-65,20)
          \qbezier[10](-45,-20)(-55,-20)(-65,-20)

          \put(-5,0){\circle*{2}}
          \put(15,20){\circle*{2}}
          \put(15,0){\circle*{2}}
          \put(5,10){\circle*{2}}
          \put(-15,10){\circle*{2}}
          \put(-5,20){\circle*{2}}
          \put(5,-10){\circle*{2}}
          \put(-15,-10){\circle*{2}}
          \put(-5,-20){\circle*{2}}
          \put(15,-20){\circle*{2}}
          \multiput(-6,-1)(0,20){2}{\vector(-1,-1){8}}
          \multiput(4,-11)(0,20){2}{\vector(-1,-1){8}}
          \multiput(14,-1)(0,20){2}{\vector(-1,-1){8}}
          \multiput(-6,1)(0,-20){2}{\vector(-1,1){8}}
          \multiput(4,-9)(0,20){2}{\vector(-1,1){8}} 
          \multiput(14,1)(0,-20){2}{\vector(-1,1){8}}
          \put(-13,15){$b$}
          \put(.2,15){$c$}
          \put(7,15){$d$}
          \put(-9.2,5){$f$}
          \put(-2.5,5){$g$}
          \put(10.1,5){$h$}
          \put(-13,-5){$j$}
          \put(.2,-5){$k$}
          \put(8.4,-5){$l$}
          \put(-9.6,-15){$n$}
          \put(-2.4,-15){$o$}
          \put(10.1,-15){$p$}
          \qbezier[10](15,20)(5,20)(-5,20)
          \qbezier[10](15,-20)(5,-20)(-5,-20)

          \multiput(-6.5,0)(0,20){2}{\vector(-1,0){37}} 
          \multiput(-16.5,-10)(0,20){2}{\vector(-1,0){37}}
          \multiput(-6.5,-20)(0,0){2}{\vector(-1,0){37}}
          \put(-26.7,20.6){$a$}
          \put(-26.7,.8){$i$}
          \put(-26.7,-19.2){$q$}
          \put(-37.7,10.6){$e$}
          \put(-37.7,-9.3){$m$}

          \put(55,0){\circle*{2}}
          \put(75,20){\circle*{2}}
          \put(75,0){\circle*{2}}
          \put(65,10){\circle*{2}}
          \put(45,10){\circle*{2}}
          \put(55,20){\circle*{2}}
          \put(65,-10){\circle*{2}}
          \put(45,-10){\circle*{2}}
          \put(55,-20){\circle*{2}}
          \put(75,-20){\circle*{2}}
          \multiput(54,-1)(0,20){2}{\vector(-1,-1){8}}
          \multiput(64,-11)(0,20){2}{\vector(-1,-1){8}}
          \multiput(74,-1)(0,20){2}{\vector(-1,-1){8}}
          \multiput(54,1)(0,-20){2}{\vector(-1,1){8}}
          \multiput(64,-9)(0,20){2}{\vector(-1,1){8}} 
          \multiput(74,1)(0,-20){2}{\vector(-1,1){8}}
          \put(47,15){$b$}
          \put(60.2,15){$c$}
          \put(67,15){$d$}
          \put(50.8,5){$f$}
          \put(57.5,5){$g$}
          \put(70.1,5){$h$}
          \put(47,-5){$j$}
          \put(60.2,-5){$k$}
          \put(68.4,-5){$l$}
          \put(50.4,-15){$n$}
          \put(57.6,-15){$o$}
          \put(70.1,-15){$p$}
          \qbezier[10](75,20)(65,20)(55,20)
          \qbezier[10](75,-20)(65,-20)(55,-20)

          \multiput(53.5,0)(0,20){2}{\vector(-1,0){37}} 
          \multiput(43.5,-10)(0,20){2}{\vector(-1,0){37}}
          \multiput(53.5,-20)(0,0){2}{\vector(-1,0){37}}
          \put(33.3,20.6){$a$}
          \put(33.3,.8){$i$}
          \put(33.3,-19.2){$q$}
          \put(23.3,10.6){$e$}
          \put(23.3,-9.3){$m$}

          \put(-73,20){$\dots$}
          \put(-83,10){$\dots$}
          \put(-73,0){$\dots$}
          \put(-83,-10){$\dots$}
          \put(-73,-20){$\dots$}

          \put(77,20){$\dots$}
          \put(67,10){$\dots$}
          \put(77,0){$\dots$}
          \put(67,-10){$\dots$}
          \put(77,-20){$\dots$}
        }     
      \end{picture}
    \end{center}
In this diagram the dotted lines indicate zero relations. 
In addition to these relations all the squares commute.

\section{Iterated tubular coverings}  \label{sec:ItTubCov}
\subsection{}
We show in this section that the four Galois coverings constructed in 
Section~\ref{sec:Gal} are iterated tubular in the sense of de la 
Peña and Tomé~\cite{DELA}. We will use the characterization of
iterated tubular algebras from Section~\ref{ssec:TubIt}. To this end
we will use Proposition~\ref{ringel} in order to identify the relevant
extension modules as simple regular modules.
This is more or less the strategy from~\cite[Sec.~2]{HaRi} where the
authors show that the repetitive algebra of a canonical algebra of tubular
type is iterated tubular.

\subsection{Covering of $\cP(Q^{(1)},W^{(1)})$} We are dealing
with the covering described in Section~\ref{cov:333}.
Consider the following tame concealed algebra, which is  of tubular type 
$(2,2,2)$, \cite[pages~365--366]{RIN}:
      
      \begin{figure}[ht]
        \begin{center}
          \setlength{\unitlength}{0.6mm}
          \begin{picture}(120,34)
            \put(65,17){
              \put(-5,15){\circle*{2}}
              \put(-5,5){\circle*{2}}
              \put(-5,-15){\circle*{2}}
              \put(5,-5){\circle*{2}}
              \put(-15,-5){\circle*{2}}
              \multiput(-6,14)(-20,0){1}{\vector(-1,-2){8.7}}
              \multiput(-6,4)(-20,0){1}{\vector(-1,-1){8}}
              \multiput(4,-4)(-10,10){1}{\vector(-1,1){8}}
              \multiput(4,-6)(-10,10){1}{\vector(-1,-1){8}}
              \multiput(-6,-14)(-20,0){1}{\vector(-1,1){8}}
              \multiput(4.3,-4)(-20,0){1}{\vector(-1,2){8.8}}
              \qbezier[12](-13.5,-5)(-5,-5)(3.5,-5)         
            }    
          \end{picture}
        \end{center}   
        \caption{}\label{f:1}
      \end{figure}

\pagebreak[3]

 Consider the modules with dimension vectors:\quad

      \begin{equation}\label{v:1}
        \begin{array}{c}
          0\\
          1\\
          1\phantom{00}1 \\
          1 
        \end{array},\quad
        \begin{array}{c}
          1\\
          0\\
          1\phantom{00}1 \\
          1
        \end{array}\text{\quad and}\quad
        \begin{array}{c}
          1\\
          1\\
          1\phantom{00}1 \\
          0
        \end{array}
      \end{equation}

      Applying the Coxeter transformation we see that: 

      \[
      \begin{pmatrix}
        0  \\
        1  \\
        1\phantom{00}1 \\
        1 
      \end{pmatrix}\Phi=
      \begin{array}{c}
        1  \\
        0 \\
        0\phantom{00}   0 \\
        0
      \end{array},\qquad
      \begin{pmatrix}
        1  \\
        0  \\
        0\phantom{00} 0 \\
        0 
      \end{pmatrix}\Phi=
      \begin{array}{c}
        0  \\
        1  \\
        1\phantom{00}1 \\
        1 
      \end{array}; 
      \]

      \[
      \begin{pmatrix}
        1  \\
        0  \\
        1\phantom{00} 1 \\
        1 
      \end{pmatrix}\Phi=
      \begin{array}{c}
        0 \\
        1  \\
        0\phantom{00}  0 \\
        0 
      \end{array},\qquad
      \begin{pmatrix}
        0  \\
        1  \\
        0\phantom{00}0 \\
        0 
      \end{pmatrix}\Phi=
      \begin{array}{c}
        1  \\
        0  \\
        1\phantom{00}1 \\
        1 
      \end{array}; 
      \]

      \[
      \begin{pmatrix}
        1 \\
        1 \\
        1\phantom{00}1 \\
        0 
      \end{pmatrix}\Phi=
      \begin{array}{ccc}
        0  \\
        0  \\
        0\phantom{00}0 \\
        1 
      \end{array},\qquad
      \begin{pmatrix}
        0  \\
        0  \\
        0\phantom{00}0 \\
        1 
      \end{pmatrix}\Phi=
      \begin{array}{ccc}
        1 \\
        1 \\
        1\phantom{00}1 \\
        0 
      \end{array}; \qquad h_1=
      \begin{array}{c}
        1  \\
        1  \\
        1\phantom{00}1 \\
        1 
      \end{array}; 
      \]

Here the vector $h_1$ is the positive generator of the radical of the 
quadratic form of the quiver with relations in Figure~\ref{f:1} . 
By Proposition~\ref{ringel} the modules with 
dimension vectors in \eqref{v:1} are simple regular modules in
different tubes of rank $2$. Applying one-point extensions with these modules 
we obtain, by~\cite{RIN}, the following tubular algebra of type $(3,3,3)$:

      \begin{center}
        \setlength{\unitlength}{0.6mm}
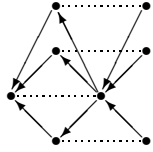
        \begin{picture}(120,40)
          \put(40,20){
            \put(15,15){\circle*{2}} 
            \put(-5,15){\circle*{2}} 
            \put(15,5){\circle*{2}}
            \put(-5,5){\circle*{2}}
            \put(15,-15){\circle*{2}}
            \put(-5,-15){\circle*{2}}
            \put(5,-5){\circle*{2}}
            \put(-15,-5){\circle*{2}}
            \multiput(14,14)(-20,0){2}{\vector(-1,-2){8.7}}
            \multiput(14,4)(-20,0){2}{\vector(-1,-1){8}}
            \multiput(14,-14)(-10,10){2}{\vector(-1,1){8}}
            \multiput(4,-6)(-10,10){1}{\vector(-1,-1){8}}
            \multiput(-6,-14)(-20,0){1}{\vector(-1,1){8}}
            \multiput(4.3,-4)(-20,0){1}{\vector(-1,2){8.8}}
            \qbezier[12](-3.5,15)(5,15)(13.5,15)
            \qbezier[12](-3.5,5)(5,5)(13.5,5)
            \qbezier[12](-3.5,-15)(5,-15)(13.5,-15)
            \qbezier[12](-13.5,-5)(-5,-5)(3.5,-5)
          }
        \end{picture} 
      \end{center}
      
Now, in the above figure we find the following tame concealed subca\-tegory, 
which is of tubular type $(3,3,2)$
(tame concealed of type $\widetilde{E}_6$)~\cite[pp.~365--366]{RIN}:

      \begin{figure}[ht]
        \begin{center}
          \setlength{\unitlength}{0.6mm}
          \begin{picture}(120,40)
            \put(40,20){
              \put(15,15){\circle*{2}}  
              \put(-5,15){\circle*{2}} 
              \put(15,5){\circle*{2}}
              \put(-5,5){\circle*{2}}
              \put(15,-15){\circle*{2}}
              \put(-5,-15){\circle*{2}}
              \put(5,-5){\circle*{2}}

              \multiput(14,14)(-20,0){1}{\vector(-1,-2){8.7}}
              \multiput(14,4)(-20,0){1}{\vector(-1,-1){8}}
              \multiput(14,-14)(-10,10){2}{\vector(-1,1){8}}
              \multiput(4,-6)(-10,10){1}{\vector(-1,-1){8}}
              \multiput(4.3,-4)(-20,0){1}{\vector(-1,2){8.8}}
              \qbezier[12](-3.5,15)(5,15)(13.5,15)
              \qbezier[12](-3.5,5)(5,5)(13.5,5)
              \qbezier[12](-3.5,-15)(5,-15)(13.5,-15)
            }
          \end{picture} 
        \end{center}
        \caption{}\label{f:2}
      \end{figure}
      
We consider the (simple) module with dimension vector:
      \begin{equation}\label{v:2}
        \begin{array}{c}
          0\phantom{00}0 \\
          0\phantom{00} 0\\
          1   \\
          0\phantom{00} 0
        \end{array}
      \end{equation}
      
Applying the Coxeter transformation we see that:       
      \[
      \begin{pmatrix}
        0\phantom{00}0 \\
        0\phantom{00} 0\\
        1   \\
        0\phantom{00} 0
      \end{pmatrix}\Phi=
      \begin{array}{c}
        1\phantom{00}1 \\
        1\phantom{00}1\\
        2   \\
        1\phantom{00}1     
      \end{array} \quad \text{and} \qquad
      \begin{pmatrix}
        1\phantom{00}1 \\
        1\phantom{00}1\\
        2   \\
        1\phantom{0}1     
      \end{pmatrix}\Phi=
      \begin{array}{c}
        0\phantom{00}0 \\
        0\phantom{00} 0\\
        1   \\
        0\phantom{00}0     
      \end{array}; \qquad h_2=
      \begin{array}{c}
        1\phantom{00}1 \\
        1\phantom{00}1\\
        3   \\
        1\phantom{00}1     
      \end{array} 
      \]
Here, the vector $h_2$ is the positive generator of the radical of the quadratic form of the quiver with relations in Figure~\ref{f:2}. 
By Proposition~\ref{ringel}, the module with 
dimension vector in \eqref{v:2} is a simple regular module in a tube of
rank $2$. The one-point extension with this module is, by~\cite{RIN}, 
the following tubular algebra of type $(3,3,3)$:

      \begin{center}
        \setlength{\unitlength}{0.6mm}
        \begin{picture}(120,40)
          \put(40,20){
            \put(15,15){\circle*{2}} 
            \put(-5,15){\circle*{2}} 
            \put(15,5){\circle*{2}}
            \put(-5,5){\circle*{2}}
            \put(15,-15){\circle*{2}}
            \put(-5,-15){\circle*{2}}
            \put(5,-5){\circle*{2}}

            \multiput(14,14)(-20,0){1}{\vector(-1,-2){8.7}}
            \multiput(14,4)(-20,0){1}{\vector(-1,-1){8}}
            \multiput(14,-14)(-10,10){2}{\vector(-1,1){8}}
            \multiput(4,-6)(-10,10){1}{\vector(-1,-1){8}}
            \multiput(4.3,-4)(-20,0){1}{\vector(-1,2){8.8}}

            \qbezier[12](-3.5,15)(5,15)(13.5,15)
            \qbezier[12](-3.5,5)(5,5)(13.5,5)
            \qbezier[12](-3.5,-15)(5,-15)(13.5,-15)

            \put(45,-5){\circle*{2}} 
            \multiput(43,-5)(-20,0){1}{\vector(-1,0){36}}

            \qbezier[25](-3.5,14.5)(20,5)(43.5,-4.7)
            \qbezier[25](-3.5,4.5)(20,0)(43.5,-4.7)
            \qbezier[25](-3.5,-14.5)(20,-10)(43.5,-5.3)
          }     
        \end{picture}
      \end{center}

Now in the above figure, we find the following tame concealed subalgebra, 
which is of tubular type $(2,2,2)$, (in fact, this is an hereditary algebra of 
type $\widetilde{D}_4$) \cite[pages~365--366]{RIN}:see Figure~\ref{f:3}

      \begin{figure}[ht]
        \begin{center}
          \setlength{\unitlength}{0.6mm}
          \begin{picture}(120,40)
            \put(30,20){
              \put(15,15){\circle*{2}} 
              \put(15,5){\circle*{2}}
              \put(15,-15){\circle*{2}}
              \put(5,-5){\circle*{2}}

              \multiput(14,14)(-20,0){1}{\vector(-1,-2){8.7}}
              \multiput(14,4)(-20,0){1}{\vector(-1,-1){8}}
              \multiput(14,-14)(-10,10){1}{\vector(-1,1){8}}

              \put(45,-5){\circle*{2}} 
              \multiput(43,-5)(-20,0){1}{\vector(-1,0){36}}
            }     
          \end{picture}
        \end{center}
        \caption{}\label{f:3}
      \end{figure}

Consider the modules with dimension vectors:
      \begin{equation}\label{v:3}
        \begin{array}{c}
          1  \\
          0 \\
          1\phantom{00}   1 \\
          0
        \end{array},\qquad
        \begin{array}{c}
          0  \\
          1 \\
          1\phantom{00}   1 \\
          0
        \end{array}\text{ \quad and}\qquad
        \begin{array}{c}
          0  \\
          0 \\
          1\phantom{00}   1 \\
          1
        \end{array}
      \end{equation}

      Applying the Coxeter transformation we see that: 

      \[
      \begin{pmatrix}
        1  \\
        0 \\
        1\phantom{00}   1 \\
        0
      \end{pmatrix}\Phi=
      \begin{array}{c}
        0  \\
        1 \\
        1\phantom{00}   0 \\
        1     
      \end{array},\qquad
      \begin{pmatrix}
        0  \\
        1  \\
        1\phantom{00} 0 \\
        1 
      \end{pmatrix}\Phi=
      \begin{array}{c}
        1  \\
        0  \\
        1\phantom{00}1 \\
        0 
      \end{array}; 
      \]
      \[
      \begin{pmatrix}
        0  \\
        1  \\
        1\phantom{00} 1 \\
        0 
      \end{pmatrix}\Phi=
      \begin{array}{c}
        1 \\
        0  \\
        1\phantom{00}  0 \\
        1 
      \end{array},\qquad
      \begin{pmatrix}
        1  \\
        0  \\
        1\phantom{00}0 \\
        1 
      \end{pmatrix}\Phi=
      \begin{array}{c}
        0  \\
        1  \\
        1\phantom{00}1 \\
        0 
      \end{array}; 
      \]

      \[
      \begin{pmatrix}
        0 \\
        0 \\
        1\phantom{00}1 \\
        1 
      \end{pmatrix}\Phi=
      \begin{array}{ccc}
        1  \\
        1  \\
        1\phantom{00}0 \\
        0 
      \end{array},\qquad
      \begin{pmatrix}
        1  \\
        1  \\
        1\phantom{00}0 \\
        0 
      \end{pmatrix}\Phi=
      \begin{array}{ccc}
        0  \\
        0  \\
        1\phantom{00}1 \\
        1 
      \end{array}; \qquad h_3=
      \begin{array}{c}
        1 \\
        1 \\
        2\phantom{00}1 \\
        1 
      \end{array}. 
      \]

Here the vector $h_3$ is the positive  generator of the radical of the 
quadratic form of  the quiver with relations in Figure~\ref{f:3}. 
By Proposition~\ref{ringel}, the modules with 
dimension vectors in \eqref{v:3} are simple regular modules in different tubes 
of rank $2$. Applying one-point extensions with these modules we obtain, 
by~\cite{RIN}, the following tubular algebra of type $(3,3,3)$:

      \begin{center}
        \setlength{\unitlength}{0.6mm}
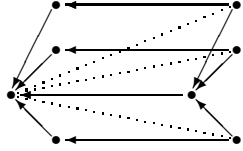
        \begin{picture}(120,50)
          \put(30,25){
            \put(15,15){\circle*{2}} 
            \put(15,5){\circle*{2}}
            \put(15,-15){\circle*{2}}
            \put(5,-5){\circle*{2}}
            \multiput(14,14)(-20,0){1}{\vector(-1,-2){8.7}}
            \multiput(14,4)(-20,0){1}{\vector(-1,-1){8}}
            \multiput(14,-14)(-10,10){1}{\vector(-1,1){8}}

            \put(55,15){\circle*{2}}
            \put(55,5){\circle*{2}}
            \put(55,-15){\circle*{2}}
            \put(45,-5){\circle*{2}}
            \multiput(54,14)(-20,0){1}{\vector(-1,-2){8.7}}
            \multiput(54,4)(-20,0){1}{\vector(-1,-1){8}}
            \multiput(54,-14)(-20,0){1}{\vector(-1,1){8}}

            \multiput(53,-15)(0,20){2}{\vector(-1,0){36}}
            \multiput(53,15)(0,0){2}{\vector(-1,0){36}}
            \multiput(43,-5)(0,0){2}{\vector(-1,0){36}}  

            \qbezier[25](6.5,-4.5)(30,5)(53.5,14.5) 
            \qbezier[25](6.5,-4.5)(30,0)(53.5,4.5)
            \qbezier[25](6.5,-5.5)(30,-10)(53.5,-14.5)
          }     
        \end{picture}
      \end{center}

Now in the above figure we find the following tame concealed subalgebra, 
which is of tubular type $(3,3,2)$, (in fact, this is an hereditary algebra of 
type $\widetilde{E}_6$) \cite[pp.~365--366]{RIN}: see Figure~\ref{f:4}

      \begin{figure}[ht]
        \begin{center}
          \setlength{\unitlength}{0.6mm}
          \begin{picture}(120,50)
            \put(25,25){
              \put(15,15){\circle*{2}} 
              \put(15,5){\circle*{2}}
              \put(15,-15){\circle*{2}}

              \put(55,15){\circle*{2}}
              \put(55,5){\circle*{2}}
              \put(55,-15){\circle*{2}}
              \put(45,-5){\circle*{2}}
              \multiput(54,14)(-20,0){1}{\vector(-1,-2){8.7}}
              \multiput(54,4)(-20,0){1}{\vector(-1,-1){8}}
              \multiput(54,-14)(-20,0){1}{\vector(-1,1){8}}

              \multiput(53,-15)(0,20){2}{\vector(-1,0){36}}
              \multiput(53,15)(0,0){1}{\vector(-1,0){36}}
            }     
          \end{picture}
        \end{center}
        \caption{}\label{f:4}
      \end{figure}

 \pagebreak[3]

 Consider the module with dimension vector:

      \begin{equation}\label{v:4}
        \begin{array}{c}
          0\phantom{00}1  \\
          0\phantom{00}1 \\
          2 \\
          0\phantom{00}1
        \end{array},\qquad
      \end{equation}
Applying the Coxeter transformation we see that: 

      \[
      \begin{pmatrix}
        0\phantom{00}1  \\
        0\phantom{00}1\\
        2 \\
        0\phantom{00}1
      \end{pmatrix}\Phi=
      \begin{array}{c}
        1\phantom{00}1  \\
        1\phantom{00}1\\
        1 \\
        1\phantom{00}1
      \end{array},\quad
      \begin{pmatrix}
        1\phantom{00}1  \\
        1\phantom{00}1\\
        1 \\
        1\phantom{00}1
      \end{pmatrix}\Phi=
      \begin{array}{c}
        0\phantom{00}1  \\
        0\phantom{00}1\\
        2 \\
        0\phantom{00}1
      \end{array}; \qquad h_4=
      \begin{array}{c}
        1\phantom{00}2\\
        1\phantom{00}2\\
        3 \\
        1\phantom{00}2
      \end{array}.
      \]
Here the vector $h_4$ is the positive  generator of the radical of the 
quadratic form of the quiver with relations in Figure~\ref{f:4}. 
By Proposition~\ref{ringel}, the module with 
dimension vector in~\eqref{v:4} is a simple regular module in a tube of
rank $2$. The one-point extension with this module is, by \cite{RIN}, 
the following tubular algebra of type $(3,3,3)$:

      \begin{center}
        \setlength{\unitlength}{0.6mm}
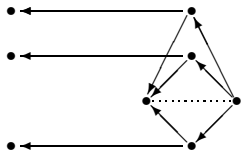
        \begin{picture}(120,50)
          \put(60,25){
            \put(-25,15){\circle*{2}} 
            \put(-25,5){\circle*{2}}
            \put(-25,-15){\circle*{2}}
            \put(15,15){\circle*{2}}
            \put(15,5){\circle*{2}}
            \put(15,-15){\circle*{2}}
            \put(25,-5){\circle*{2}}
            \put(5,-5){\circle*{2}}
            \multiput(14,14)(-20,0){1}{\vector(-1,-2){8.7}}
            \multiput(14,4)(-20,0){1}{\vector(-1,-1){8}}
            \multiput(24,-4)(-10,10){1}{\vector(-1,1){8}}
            \multiput(24,-6)(-10,10){1}{\vector(-1,-1){8}}
            \multiput(14,-14)(-20,0){1}{\vector(-1,1){8}}
            \multiput(24.3,-4)(-20,0){1}{\vector(-1,2){8.8}}
            \multiput(13,-15)(0,20){2}{\vector(-1,0){36}}
            \multiput(13,15)(0,0){2}{\vector(-1,0){36}}

            \qbezier[12](6.5,-5)(15,-5)(23.5,-5)
          }     
        \end{picture}
      \end{center}

In the above figure, we find the same subalgebra of tubular type 
$(2,2,2,2)$ as the one in Figure \ref{f:1}. 
So, this Galois covering is an iterated 
tubular algebra in the sense of de la Peña-Tomé \cite{DELA}
      \medskip

\subsection{Covering of $\cP(Q^{(2)},W^{(2)}_\lam)$} We are dealing
with the covering described in Section~\ref{cov:2222}. This case
is well-known, we include it for completeness.

Consider the following tame concealed (in fact hereditary) algebra, 
which is of tubular type $(2,2)$, \cite[pp.~365--366]{RIN}:

      \begin{figure}[ht]
        \begin{center} 
          \setlength{\unitlength}{0.7mm}
          \begin{picture}(80,34)
            \put(40,17){
              \put(10,10){\circle*{2}} 
              \put(10,-10){\circle*{2}}
              \put(-10,10){\circle*{2}}
              \put(-10,-10){\circle*{2}}
              \multiput(8,10)(0,-20){2}{\vector(-1,0){16}}

              \multiput(8.5,8.5)(-20,0){1}{\vector(-1,-1){17}}
              \multiput(8.5,-8.5)(-20,0){1}{\vector(-1,1){17}}
            }     
          \end{picture}
        \end{center}
        \caption{}\label{f:5}
      \end{figure}
The vector
      \begin{equation}\label{v:5}
        h_5=
        \begin{array}{c}       
          1\phantom{,,}1 \\                 
          1\phantom{,,}1       
        \end{array}
      \end{equation}
is the positive radical generator for the hereditary algebra from
Figure~\ref{f:5}. We have an obvious one parameter family 
$(M_\lam)_{\lam\in K\setminus\{0\}}$ of indecomposable 
(simple regular) modules with this dimension vector. One point extension
with $M_1$ and $M_\lam$ yields by~\cite{RIN} 
a tubular algebra of type $(2,2,2,2)$ with the following quiver:

      \begin{center} 
        \setlength{\unitlength}{0.6mm}
        \begin{picture}(80,34)
          \put(30,17){
            \put(10,10){\circle*{2}} 
            \put(10,-10){\circle*{2}}
            \put(-10,10){\circle*{2}}
            \put(-10,-10){\circle*{2}}
            \put(30,10){\circle*{2}}
            \put(30,-10){\circle*{2}}
            \multiput(28,10)(0,-20){2}{\vector(-1,0){16}}
            \multiput(8,10)(0,-20){2}{\vector(-1,0){16}}
            \multiput(8.5,8.5)(20,0){2}{\vector(-1,-1){17}}
            \multiput(8.5,-8.5)(20,0){2}{\vector(-1,1){17}}
          }     
        \end{picture}
      \end{center}

\begin{remark}
In the above figure three squares commute, and fourth one commutes up to 
a factor $\lambda$.
\end{remark}

This tubular algebra is also a cotubular algebra of a hereditary algebra
of tubular type $(2,2)$, isomorphic to the one in Figure~\ref{f:5}. 
We conclude by Section~\ref{ssec:TubIt} that this Galois covering is an iterated tubular 
algebra in the sense of de la Peña-Tomé~\cite{DELA}. Notice that in this
case we are dealing with a self-injective Galois covering.

\subsection{Covering of $\cP(Q^{(3)},W^{(3)})$} We are dealing
with the covering described in Section~\ref{cov:442}.
Consider the  tame concealed algebra of tubular 
type $(4,2,2)$ which is displayed in Figure~\ref{f:7}, 
\cite[pages~365--366]{RIN}:

      \begin{figure}[ht]
        \begin{center}
          \setlength{\unitlength}{0.6mm}
          \begin{picture}(120,50)
            \put(25,25){         
              \put(0,0){\circle*{2}}
              \put(10,10){\circle*{2}}
              \put(-10,10){\circle*{2}}
              \put(0,20){\circle*{2}}
              \put(10,-10){\circle*{2}}
              \put(-10,-10){\circle*{2}}
              \put(0,-20){\circle*{2}}
              \multiput(-1,1)(10,10){2}{\vector(-1,1){8}}
              \multiput(9,-9)(10,10){1}{\vector(-1,1){8}}
              \multiput(-1,-19)(10,-10){1}{\vector(-1,1){8}}

              \multiput(-1,19)(10,10){1}{\vector(-1,-1){8}}
              \multiput(-1,-1)(10,10){2}{\vector(-1,-1){8}} 
              \multiput(9,-11)(10,10){1}{\vector(-1,-1){8}}
              \qbezier[12](-10,10)(0,10)(10,10)
              \qbezier[12](-10,-10)(0,-10)(10,-10)
            }     
          \end{picture}
        \end{center}
        \caption{}\label{f:7}
      \end{figure}

By the same kind of argument as in the previous cases we see that the
indecomposable modules for this quiver with relations and with 
dimension vectors:\quad

      \begin{equation}\label{v:7}
        u_1=
        \begin{array}{c}
          1 \\
          1\phantom{,,,}1 \\
          1 \\
          1\phantom{,,,}1\\
          1
        \end{array},\text{\quad and\quad} 
        v_1=
        \begin{array}{c}
          0 \\
          0\phantom{,,,}0\\
          1 \\
          0\phantom{,,,}0\\
          0
        \end{array}
      \end{equation}
are simple regular and belong to the same tube of rank $2$. One point extension
with these two modules yields the following tubular algebra of type $(4,4,2)$:

      \begin{center}
        \setlength{\unitlength}{0.6mm}
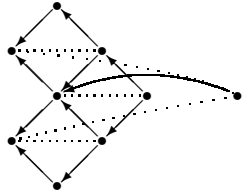
        \begin{picture}(120,50)
          \put(35,25){         
            \put(-10,0){\circle*{2}} 
            \put(10,0){\circle*{2}}
            \put(0,10){\circle*{2}}
            \put(-20,10){\circle*{2}}
            \put(-10,20){\circle*{2}}
            \put(0,-10){\circle*{2}}
            \put(-20,-10){\circle*{2}}
            \put(-10,-20){\circle*{2}}
            \put(30,0){\circle*{2}}

            \multiput(-11,1)(10,10){2}{\vector(-1,1){8}}
            \multiput(-1,-9)(10,10){2}{\vector(-1,1){8}}
            \multiput(-11,-19)(10,-10){1}{\vector(-1,1){8}}

            \multiput(-11,19)(10,10){1}{\vector(-1,-1){8}}
            \multiput(-11,-1)(10,10){2}{\vector(-1,-1){8}} 
            \multiput(-1,-11)(10,10){2}{\vector(-1,-1){8}}

            \qbezier[12](-20,10)(-10,10)(0,10)
            \qbezier[12](-20,-10)(-10,-10)(0,-10)
            \qbezier[12](-10,0)(0,0)(10,0)

            \qbezier[22](-20,10)(20,6)(30,0)
            \qbezier[22](-20,-10)(0,-5)(30,0)

            \qbezier(28,1)(10,8)(-7,1.4)
            \put(-2.5,2.9){\vector(-3,-1){6}}
          }     
        \end{picture}
      \end{center}

Now, in the above figure we find the following tame concealed subalgebra, 
which is of tubular type $(3,3,2)$, \cite[pp.~365--366]{RIN}:

      \begin{figure}[ht]
        \begin{center}
          \setlength{\unitlength}{0.6mm}
          \begin{picture}(120,50)
            \put(46,25){
              \put(-20,0){\circle*{2}} 
              \put(0,0){\circle*{2}}
              \put(-10,10){\circle*{2}}
              \put(-20,20){\circle*{2}}
              \put(-10,-10){\circle*{2}}
              \put(-20,-20){\circle*{2}}
              \put(20,0){\circle*{2}}
              \multiput(-11,11)(10,10){1}{\vector(-1,1){8}}
              \multiput(-11,-9)(10,10){2}{\vector(-1,1){8}}
              \multiput(-11,9)(10,10){1}{\vector(-1,-1){8}} 
              \multiput(-11,-11)(10,10){2}{\vector(-1,-1){8}}
              \qbezier(18,1)(0,8)(-17,1.4)
              \put(-12.5,2.9){\vector(-3,-1){6}}
              \qbezier[12](-20,0)(-10,0)(0,0)
            }     
          \end{picture}
        \end{center}
        \caption{}\label{f:8}
      \end{figure}

The indecomposable modules with dimension vectors:

      \begin{equation}\label{v:8}
        u_1=
        \begin{array}{l}
          0 \\
          \phantom{00} 1 \\
          1 \phantom{,00}0\phantom{,00}1\\
          \phantom{00}0\\
          0
        \end{array}\text{\quad and\quad}
        v_1=
        \begin{array}{l}
          0 \\
          \phantom{00} 0 \\
          1 \phantom{,00}0\phantom{,00}1\\
          \phantom{00}1\\
          0
        \end{array}
      \end{equation}
are simple regular in different tubes of rank $3$.

Applying one-point extensions with these modules we obtain, 
by~\cite{RIN}, the following tubular algebra of type $(4,4,2)$:

      \begin{center}
        \setlength{\unitlength}{0.6mm}
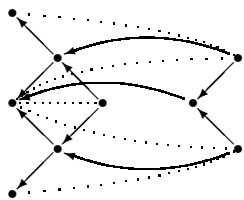
        \begin{picture}(120,50)
          \put(55,25){
            \put(-25,0){\circle*{2}} 
            \put(-5,0){\circle*{2}}
            \put(-15,10){\circle*{2}}
            \put(-25,20){\circle*{2}}
            \put(-15,-10){\circle*{2}}
            \put(-25,-20){\circle*{2}}
            \put(25,10){\circle*{2}}
            \put(15,0){\circle*{2}}
            \put(25,-10){\circle*{2}}
            \qbezier(23,11)(5,18)(-12,11.4) 
            \qbezier(13,1)(-5,8)(-22,1.4)
            \multiput(-16,11)(10,10){1}{\vector(-1,1){8}}
            \multiput(-16,-9)(10,10){2}{\vector(-1,1){8}}

            \multiput(-16,9)(10,10){1}{\vector(-1,-1){8}} 
            \multiput(-16,-11)(10,10){2}{\vector(-1,-1){8}}
            \put(-7.5,12.9){\vector(-3,-1){6}}
            \put(-17.5,2.9){\vector(-3,-1){6}}
            \multiput(24,9)(10,10){1}{\vector(-1,-1){8}} 
            \multiput(24,-9)(10,10){1}{\vector(-1,1){8}} 
            \qbezier(23,-11)(5,-18)(-12,-11.4)
            \put(-7.5,-12.9){\vector(-3,1){6}}
            \qbezier[22](-25,0)(-15,10)(25,10)
            \qbezier[12](-25,0)(-15,0)(-5,0)
            \qbezier[22](-25,0)(-5,-10)(25,-10)
            \qbezier[22](-25,20)(15,17)(25,10)
            \qbezier[22](-25,-20)(15,-17)(25,-10)
          }     
        \end{picture}
      \end{center}

Now, in the above figure we find the following tame concealed 
(in fact hereditary) subalgebra, which is of tubular type $(3,3)$,
      \cite[pages~365--366]{RIN}:

      \begin{figure}[ht]
        \begin{center} 
          \setlength{\unitlength}{0.6mm}
          \begin{picture}(120,50)
            \put(60,25){
              \put(-10,0){\circle*{2}} 
              \put(-20,10){\circle*{2}}
              \put(-20,-10){\circle*{2}}
              \multiput(-11,1)(10,10){1}{\vector(-1,1){8}}
              \multiput(-11,-1)(10,10){1}{\vector(-1,-1){8}}
              \put(10,0){\circle*{2}}
              \put(20,10){\circle*{2}}
              \put(20,-10){\circle*{2}}
              \multiput(19,9)(10,10){1}{\vector(-1,-1){8}}
              \multiput(19,-9)(10,-10){1}{\vector(-1,1){8}}
              \qbezier(18,11)(0,18)(-17,11.4) 
              \qbezier(18,-11)(0,-18)(-17,-11.8)
              \put(-12.5,12.9){\vector(-3,-1){6}}
              \put(-12.5,-13.3){\vector(-3,1){6}}
            }     
          \end{picture}
        \end{center}
        \caption{}\label{f:9}
      \end{figure}
      
Consider the modules with dimension vectors:\quad

      \begin{equation}\label{v:9}
        u_1=
        \begin{array}{l}
          1\phantom{000000,}1 \\
          \phantom{00}1\phantom{,0}1\\
          1\phantom{000000,}1
        \end{array}, \quad
        v_1=
        \begin{array}{l}
          0\phantom{000000,}1 \\
          \phantom{00}0\phantom{,0}1\\
          0\phantom{000000,}0
        \end{array}\text{\quad and \quad}
        w_1=
        \begin{array}{l}
          0\phantom{000000,}0 \\
          \phantom{00}0\phantom{,0}1\\
          0\phantom{000000,}1
        \end{array}
      \end{equation}

The indecomposable module with dimension vector $u_1$ (in \ref{v:9}) 
is a simple module in a tube of rank $1$, and the modules with dimension 
vectors $v_1$ and $w_1$ 
(in~\ref{v:9}) are simple regular modules in different tubes of rank $3$. 
Applying one-point extensions with these modules we obtain
the following tubular algebra of type $(4,4,2)$:
      
      \begin{center} 
        \setlength{\unitlength}{0.6mm}
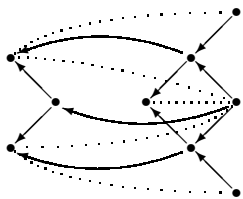
        \begin{picture}(120,50)
          \put(65,25){
            \put(-15,0){\circle*{2}} 
            \put(-25,10){\circle*{2}}
            \put(-25,-10){\circle*{2}}
            \multiput(-16,1)(10,10){1}{\vector(-1,1){8}}
            \multiput(-16,-1)(10,10){1}{\vector(-1,-1){8}}
            \put(25,0){\circle*{2}} 
            \put(5,0){\circle*{2}}
            \put(15,10){\circle*{2}}
            \put(25,20){\circle*{2}}
            \put(15,-10){\circle*{2}}
            \put(25,-20){\circle*{2}}
            \multiput(24,1)(10,10){1}{\vector(-1,1){8}}
            \multiput(14,-9)(10,-10){2}{\vector(-1,1){8}}
            \multiput(14,9)(10,10){2}{\vector(-1,-1){8}}
            \multiput(24,-1)(10,10){1}{\vector(-1,-1){8}} 
            \qbezier(13,11)(-5,18)(-22,11.4) 
            \qbezier(23,-1)(5,-8)(-12,-1.8)
            \qbezier(13,-11)(-5,-18)(-22,-11.8)

            \put(-17.5,12.9){\vector(-3,-1){6}}
            \put(-17.5,-13.3){\vector(-3,1){6}}
            \put(-7.5,-3.3){\vector(-3,1){6}}
            \qbezier[22](-25,10)(-5,10)(25,0)
            \qbezier[12](5,0)(15,0)(25,0)
            \qbezier[22](-25,-10)(15,-10)(25,0)
            \qbezier[22](-25,10)(-15,20)(25,20)
            \qbezier[22](-25,-10)(-15,-20)(25,-20)
          }     
        \end{picture}
      \end{center}

Now, in the above figure we find the following tame concealed subalgebra, 
which is of tubular type $(3,3,2)$, \cite[pp.~365--366]{RIN}:

      \begin{figure}[ht]
        \begin{center} 
          \setlength{\unitlength}{0.6mm}
          \begin{picture}(120,50)
            \put(60,25){
              \put(-20,0){\circle*{2}} 
              \put(20,0){\circle*{2}} 
              \put(0,0){\circle*{2}}
              \put(10,10){\circle*{2}}
              \put(20,20){\circle*{2}}
              \put(10,-10){\circle*{2}}
              \put(20,-20){\circle*{2}}
              \multiput(19,1)(10,10){1}{\vector(-1,1){8}}
              \multiput(9,-9)(10,-10){2}{\vector(-1,1){8}}
              \multiput(9,9)(10,10){2}{\vector(-1,-1){8}}
              \multiput(19,-1)(10,10){1}{\vector(-1,-1){8}} 
              \qbezier(18,-1)(0,-8)(-17,-1.8) 
              \put(-12.5,-3.3){\vector(-3,1){6}}
              \qbezier[12](0,0)(10,0)(20,0)
            }     
          \end{picture}
        \end{center}
        \caption{}\label{f:10}
      \end{figure}

Consider the indecomposable modules with dimension vectors:\quad

      \begin{equation}\label{v:10}
        u_1=
        \begin{array}{r}
          1\\
          1\phantom{0,} \\
          0\phantom{,,,}1\phantom{,,,}1\\
          1\phantom{0,}\\
          0
        \end{array}\text{\quad and \quad}
        v_1=
        \begin{array}{r}
          0\\
          1\phantom{0,} \\
          0\phantom{,,,}1\phantom{,,,}1\\
          1\phantom{0,}\\
          1
        \end{array}
      \end{equation}

They are simple regular modules in different 
tubes of rank $3$. Applying one-point extensions with these modules we obtain
the following tubular algebra of type $(4,4,2)$:

      \begin{center} 
        \setlength{\unitlength}{0.6mm}
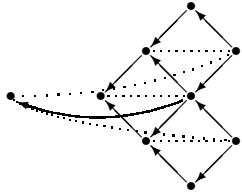
        \begin{picture}(120,50)
          \put(65,25){
            \put(-25,0){\circle*{2}}
            \put(15,0){\circle*{2}}
            \put(-5,0){\circle*{2}}

            \put(25,10){\circle*{2}}
            \put(5,10){\circle*{2}}
            \put(15,20){\circle*{2}}
            \put(25,-10){\circle*{2}}
            \put(5,-10){\circle*{2}}
            \put(15,-20){\circle*{2}}
            \multiput(14,1)(10,10){2}{\vector(-1,1){8}}
            \multiput(24,-9)(10,10){1}{\vector(-1,1){8}}
            \multiput(4,-9)(10,-10){2}{\vector(-1,1){8}}

            \multiput(4,9)(10,10){2}{\vector(-1,-1){8}}
            \multiput(14,-1)(10,10){2}{\vector(-1,-1){8}} 
            \multiput(24,-11)(10,10){1}{\vector(-1,-1){8}}
            \qbezier(13,-1)(-5,-8)(-22,-1.8)  
            \put(-17.5,-3.3){\vector(-3,1){6}}  
            \qbezier[12](5,10)(15,10)(25,10)
            \qbezier[12](-5,0)(5,0)(15,0)
            \qbezier[12](5,-10)(15,-10)(25,-10)
            \qbezier[22](-25,0)(5,0)(25,10)
            \qbezier[32](-25,-.5)(-15,-7)(25,-10)
          }     
        \end{picture}
      \end{center}

In the above figure, we find the same subalgebra of tubular type $(4,4,2)$ 
as in Figure \ref{f:7}. So, this Galois covering is an iterated tubular 
algebra in the sense of de la Peña-Tomé~\cite{DELA}.

\subsection{Covering of $\cP(Q^{(4)},W^{(4)})$} We are dealing
with the covering described in Section~\ref{cov:632}. Since it is
similar to the other cases, we skip most of the details, and display only
the relevant sequence of tame concealed and tubular subalgebras.

We start with  a tame concealed algebra $A_0$, 
which is  of tubular type $(4,2,2)$, see Figure~\ref{f:12}.
      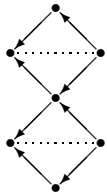
\begin{figure}[ht]
        \begin{center} 
          \setlength{\unitlength}{0.6mm}
          \begin{picture}(120,50)
            \put(35,25){
              \put(-5,20){\circle*{2}}
              \put(-15,10){\circle*{2}}
              \put(5,10){\circle*{2}}
              \put(-5,0){\circle*{2}}
              \put(-15,-10){\circle*{2}}
              \put(5,-10){\circle*{2}}
              \put(-5,-20){\circle*{2}}
              \multiput(-6,-1)(0,20){2}{\vector(-1,-1){8}}
              \multiput(4,-11)(0,20){2}{\vector(-1,-1){8}}
              \multiput(-6,1)(0,-20){2}{\vector(-1,1){8}}
              \multiput(4,-9)(0,20){2}{\vector(-1,1){8}} 
              \qbezier[12](-15,10)(-5,10)(5,10)
              \qbezier[12](-15,-10)(-5,-10)(5,-10)
            }     
          \end{picture}
        \end{center}
        \caption{$A_0$ tame concealed of tubular type $(4,4,2)$} \label{f:12}
      \end{figure}
$A_0$ admits a tubular extension $\Lam_1$  of type $(6,3,2)$, 
see Figure~\ref{f:13}.
     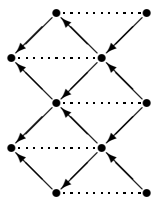
\begin{figure}[ht]
      \begin{center} 
        \setlength{\unitlength}{0.6mm}
        \begin{picture}(120,50)
          \put(35,25){
            \put(-5,0){\circle*{2}}
            \put(15,20){\circle*{2}}
            \put(15,0){\circle*{2}}
            \put(5,10){\circle*{2}}
            \put(-15,10){\circle*{2}}
            \put(-5,20){\circle*{2}}
            \put(5,-10){\circle*{2}}
            \put(-15,-10){\circle*{2}}
            \put(-5,-20){\circle*{2}}
            \put(15,-20){\circle*{2}}
            \multiput(-6,-1)(0,20){2}{\vector(-1,-1){8}}
            \multiput(4,-11)(0,20){2}{\vector(-1,-1){8}}
            \multiput(14,-1)(0,20){2}{\vector(-1,-1){8}}
            \multiput(-6,1)(0,-20){2}{\vector(-1,1){8}}
            \multiput(4,-9)(0,20){2}{\vector(-1,1){8}} 
            \multiput(14,1)(0,-20){2}{\vector(-1,1){8}}
            \qbezier[12](-5,20)(5,20)(15,20)
            \qbezier[12](-15,10)(-5,10)(5,10)
            \qbezier[12](-5,0)(5,0)(15,0)
            \qbezier[12](-15,-10)(-5,-10)(5,-10)
            \qbezier[12](-5,-20)(5,-20)(15,-20)
          }     
        \end{picture}
      \end{center}
      \caption{$\Lam_1$ tubular of type $(6,3,2)$} \label{f:13}
     \end{figure}
Now, $\Lam_1$ is a tubular coextension of a tame concealed subalgebra $A_1$,  
which is of tubular type $(4,3,2)$, see Figure~\ref{f:14}.
      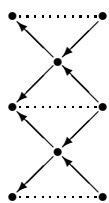
\begin{figure}[ht]
        \begin{center} 
          \setlength{\unitlength}{0.6mm}
          \begin{picture}(120,50)
            \put(35,25){
              \put(-5,20){\circle*{2}}
              \put(15,20){\circle*{2}}
              \put(5,10){\circle*{2}}
              \put(-5,0){\circle*{2}}
              \put(15,0){\circle*{2}}
              \put(5,-10){\circle*{2}}
              \put(-5,-20){\circle*{2}}
              \put(15,-20){\circle*{2}}
              \multiput(4,-11)(0,20){2}{\vector(-1,-1){8}}
              \multiput(14,-1)(0,20){2}{\vector(-1,-1){8}}
              \multiput(4,-9)(0,20){2}{\vector(-1,1){8}} 
              \multiput(14,1)(0,-20){2}{\vector(-1,1){8}}
              \qbezier[12](-5,20)(5,20)(15,20) 
              \qbezier[12](-5,0)(5,0)(15,0)
              \qbezier[12](-5,-20)(5,-20)(15,-20)        
            }     
          \end{picture}
        \end{center} 
        \caption{$A_1$ tame concealed of tubular type $(4,3,2)$} \label{f:14}
      \end{figure}
Next, $A_1$ admits the tubular extension $\Lam_2$ of of type $(6,3,2)$,
see Figure~\ref{f:15}.
     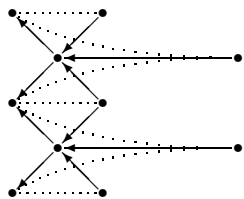
\begin{figure}[ht]
      \begin{center} 
        \setlength{\unitlength}{0.6mm}
        \begin{picture}(120,60)
          \put(35,30){
            \put(-5,20){\circle*{2}}
            \put(15,20){\circle*{2}}
            \put(5,10){\circle*{2}}
            \put(-5,0){\circle*{2}}
            \put(15,0){\circle*{2}}
            \put(5,-10){\circle*{2}}
            \put(-5,-20){\circle*{2}}
            \put(15,-20){\circle*{2}}
            \multiput(4,-11)(0,20){2}{\vector(-1,-1){8}}
            \multiput(14,-1)(0,20){2}{\vector(-1,-1){8}}
            \multiput(4,-9)(0,20){2}{\vector(-1,1){8}} 
            \multiput(14,1)(0,-20){2}{\vector(-1,1){8}}
            \put(45,10){\circle*{2}} 
            \put(45,-10){\circle*{2}}
            \multiput(43.5,-10)(0,20){2}{\vector(-1,0){37}} 
            \qbezier[12](-5,20)(5,20)(15,20) 
            \qbezier[12](-5,0)(5,0)(15,0)
            \qbezier[12](-5,-20)(5,-20)(15,-20)         
            \qbezier[22](-5,20)(10,10)(45,10)
            \qbezier[22](-5,0)(10,10)(45,10)
            \qbezier[22](-5,0)(10,-10)(45,-10)
            \qbezier[22](-5,-20)(10,-10)(45,-10)
          }     
        \end{picture}
      \end{center}
      \caption{$\Lam_2$ tubular of type $(6,3,2)$} \label{f:15}
      \end{figure}
$\Lam_2$ is a tubular coextension of the  tame concealed 
subalgebra $A_2$,  which is of tubular type $(4,2,2)$, see Figure~\ref{f:16}.

      \begin{figure}[ht]
        \begin{center} 
          \setlength{\unitlength}{0.6mm}
          \begin{picture}(120,50)
            \put(35,25){
              \put(15,20){\circle*{2}}
              \put(5,10){\circle*{2}}
              \put(15,0){\circle*{2}}
              \put(5,-10){\circle*{2}}
              \put(15,-20){\circle*{2}}
              \multiput(14,-1)(0,20){2}{\vector(-1,-1){8}}
              \multiput(14,1)(0,-20){2}{\vector(-1,1){8}}
              \put(45,10){\circle*{2}}
              \put(45,-10){\circle*{2}}
              \multiput(43.5,-10)(0,20){2}{\vector(-1,0){37}}
            }     
          \end{picture}
        \end{center}
        \caption{$A_2$ tame concealed of type $(4,2,2)$} \label{f:16}
      \end{figure}
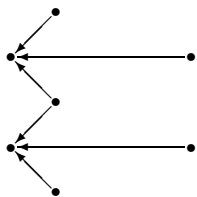
$A_2$ admits a tubular extension $\Lam_3$ of type $(6,3,2)$, 
see Figure~\ref{f:17}.
      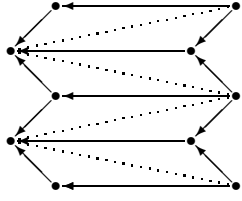
\begin{figure}[ht]
      \begin{center} 
        \setlength{\unitlength}{0.6mm}
        \begin{picture}(120,50)
          \put(35,25){
            \put(15,20){\circle*{2}}
            \put(5,10){\circle*{2}}
            \put(15,0){\circle*{2}}
            \put(5,-10){\circle*{2}}
            \put(15,-20){\circle*{2}}
            \multiput(14,-1)(0,20){2}{\vector(-1,-1){8}}
            \multiput(14,1)(0,-20){2}{\vector(-1,1){8}}
            \put(55,20){\circle*{2}}         
            \put(45,10){\circle*{2}}
            \put(55,0){\circle*{2}}
            \put(45,-10){\circle*{2}}
            \put(55,-20){\circle*{2}}
            \multiput(54,-1)(0,20){2}{\vector(-1,-1){8}}
            \multiput(54,1)(0,-20){2}{\vector(-1,1){8}}

            \multiput(53.5,0)(0,20){2}{\vector(-1,0){37}} 
            \multiput(43.5,-10)(0,20){2}{\vector(-1,0){37}}
            \multiput(53.5,-20)(0,0){2}{\vector(-1,0){37}}

            \qbezier[26](5,10)(30,15)(55,20)
            \qbezier[26](5,10)(30,5)(55,0)
            \qbezier[26](5,-10)(30,-5)(55,0)
            \qbezier[26](5,-10)(30,-15)(55,-20)
          }     
        \end{picture}
      \end{center}
      \caption{$\Lam_3$ tubular of type $(6,3,2)$} \label{f:17}
      \end{figure}

$\Lam_3$ is tubular co-extension of the tame concealed subalgebra $A_3$ 
which is of tubular type $(4,3,2)$, see Figure~\ref{f:18}.
      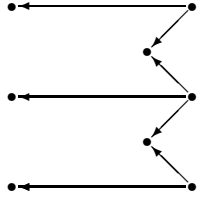
\begin{figure}[ht]
        \begin{center} 
          \setlength{\unitlength}{0.6mm}
          \begin{picture}(120,50)
            \put(70,25){
              \put(-20,20){\circle*{2}}
              \put(-20,0){\circle*{2}}
              \put(-20,-20){\circle*{2}}

              \put(20,20){\circle*{2}}
              \put(10,10){\circle*{2}}
              \put(20,0){\circle*{2}}
              \put(10,-10){\circle*{2}}
              \put(20,-20){\circle*{2}}
              \multiput(19,-1)(0,20){2}{\vector(-1,-1){8}}
              \multiput(19,1)(0,-20){2}{\vector(-1,1){8}}
              \multiput(18.5,0)(0,20){2}{\vector(-1,0){37}} 
              \multiput(18.5,-20)(0,0){2}{\vector(-1,0){37}}
            }     
          \end{picture}
        \end{center}
        \caption{$A_3$ tame concealed of tubular type $(4,3,2)$} \label{f:18}
      \end{figure}
$A_3$ admits a tubular extension $\Lam_4$ of type $(6,3,2)$, 
see Figure~\ref{f:19}.

     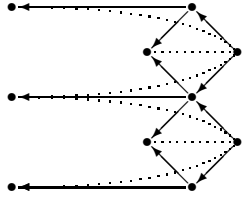
\begin{figure}[ht]
      \begin{center} 
        \setlength{\unitlength}{0.6mm}
        \begin{picture}(120,50)
          \put(35,25){
            \put(15,20){\circle*{2}}
            \put(15,0){\circle*{2}}
            \put(15,-20){\circle*{2}}

            \put(55,20){\circle*{2}}
            \put(45,10){\circle*{2}}
            \put(65,10){\circle*{2}}
            \put(55,0){\circle*{2}}
            \put(45,-10){\circle*{2}}
            \put(65,-10){\circle*{2}}
            \put(55,-20){\circle*{2}}
            \multiput(54,-1)(0,20){2}{\vector(-1,-1){8}}
            \multiput(64,-11)(0,20){2}{\vector(-1,-1){8}}
            \multiput(54,1)(0,-20){2}{\vector(-1,1){8}}
            \multiput(64,-9)(0,20){2}{\vector(-1,1){8}} 

            \multiput(53.5,0)(0,20){2}{\vector(-1,0){37}} 
            \multiput(53.5,-20)(0,0){2}{\vector(-1,0){37}}

            \qbezier[12](45,10)(55,10)(65,10)
            \qbezier[12](45,-10)(55,-10)(65,-10)

            \qbezier[26](15,20)(55,20)(65,10)
            \qbezier[26](15,0)(55,0)(65,10)
            \qbezier[26](15,0)(55,0)(65,-10)
            \qbezier[26](15,-20)(55,-20)(65,-10)
          }     
        \end{picture}
      \end{center}
      \caption{$\Lam_4$ tubular of type $(6,3,2)$} \label{f:19}
      \end{figure}
Finally, $\Lam_4$ is a cotubular extension of  a subalgebra $A_4$ of type 
$(4,2,2)$ which is equivalent to $A_0$ and we can repeat the sequence of
extensions.  We conclude that this covering is iterated tubular.

\section{Proof of Theorem~\ref{thm:main}} 
First, we observe the following consequence of our considerations:
Each of Jacobian algebras associated to the quivers with potentials 
$(Q^{(1)},W^{(1)})$, $(Q^{(2)},W^{(2)}_\lam)$, $(Q^{(3)},W^{(3)})$ and $(Q^{(4)},W^{(4)})$
is tame of polynomial growth.

Indeed, from Section~\ref{sec:ItTubCov} we know that each of these algebras
has a Galois covering  which is an iterated tubular algebra. In particular
these coverings are tame of polynomial growth  
by~\cite[Section 2.4]{DELA}, see also~\ref{ssec:TubIt}. 
Moreover, each of these Galois coverings is defined in terms of a free 
$\mathbb{Z}$-action. Thus, by Theorem~\ref{DS}, our claim follows.

Now, by~Proposition~\ref{prp:tubjac}, the endomorphism
ring of each basic cluster tilting object in a tubular cluster category
is given by a non-degenerate potential, and 
for a fixed tubular type all these algebras are connected via finite sequences of
QP mutations. By~\cite[Thm.~3.6]{GeLaSc} and the above 
observation it follows now that all these algebras
are tame of polynomial growth. Note, that for a single mutation
under the corresponding mutation operation for 
representations~\cite[Sec.~10]{DWZ1} only one component of the dimension vector 
changes. Moreover, it is easy to write down an upper bound for this new 
component which depends linearly on the dimension vector. With this observation
it is straight forward to sharpen the proof of~\cite[Thm.~3.6]{GeLaSc} 
in the sense that QP-mutation preserves even polynomial growth for tame 
algebras. This concludes the proof of Theorem~\ref{thm:main}.


\begin{thebibliography}{99}
\bibitem{BACH} 
M. Barot, Ch.~Geiss:
\emph{Tubular cluster algebras I: categorification.} 
Math.~Z. \textbf{271} (2012), no. 3-4, 1091--1115.
      
\bibitem{BaKuLe} 
M. Barot, D. Kussin, H. Lenzing:
\emph{The cluster category of a canonical algebra}, 
Trans. Amer. Math. Soc. \textbf{362} (2010), no. 8, 4313--4330.
      
\bibitem{BIRS} 
A. Buan, O. Iyama, I. Reiten, D. Smith:
\emph{Mutation of cluster-tilting objects and potentials,}  
 Amer. J. Math. \textbf{133} (2011), no. 4, 835--887.

\bibitem{DWZ1} 
H. Derksen, J. Weyman, A. Zelevinsky:
\emph{Quivers with potentials and their representations I: Mutations.}, 
Selecta Math. \textbf{14}  (2008), no. 1, 59--119.

\bibitem{DWZII} 
H. Derksen, J. Weyman, and A. Zelevinsky:
\emph{Quivers with potentials and their representations II: Applications to cluster algebras.},
J. Amer. Math. Soc. \textbf{23} (2010), 749--790.

\bibitem{DELA} 
J.A. de la Peña, B. Tomé:
\emph{Iterated Tubular Algebras}, 
J.  Pure Appl. Algebra \textbf{64} (1990), 303--314.

\bibitem{DRIN} 
V. Dlab, C. M. Ringel: 
\emph{Indecomposable representations of graphs and algebras.} 
Mem. Amer. Math. Soc. \textbf{173} (1976). 
    
\bibitem{Dosko} 
P. Dowbor, A. Skowroński:
\emph{On Galois coverings of tame algebras}, 
Arch. Math. \textbf{44} (1985), 522--529.

\bibitem{Gabriel} 
P. Gabriel:
\emph{The universal cover of a representation-finite algebra.}
Proc. Puebla 1980, Springer Lect. Notes \textbf{903}, 68--105.

\bibitem{GeLe} 
W. Geigle, H. Lenzing:
\emph{A class of weighted projective curves arising in representation theory of finite-dimensional algebras.}  In:
Singularities, representation of algebras, and vector bundles 
(Lambrecht, 1985), 265--297, Lecture Notes in Math. \textbf{1273}, 
Springer, Berlin, 1987.

\bibitem{GeKr} 
Ch. Geiss, H. Krause: 
\emph{On the notion of derived tameness.}  
J. Algebra Appl. \textbf{1} (2002), 133--158.

\bibitem{GeKeOp} 
Ch. Geiss, B. Keller, S. Oppermann:
\emph{$n$-angulated categories}, 
J. Reine Angew. Math. \textbf{675} (2013), 101--120.

\bibitem{GeLaSc} 
Ch. Geiss, D. Labardini-Fragoso, J. Schröer:
\emph{The representation type of Jacobian algebras.} 
Preprint \url{arXiv:1308.0478v1 [math.RT]}, 67pp. 
    
\bibitem{HaRi} D. Happel, C. M. Ringel: 
\emph{The derived category of a tubular algebra.}  In: 
Representation theory  I (Ottawa, Ont., 1984), pages 156-180.  
Springer, Berlin, 1986.

\bibitem{Ke} 
B. Keller:      
\emph{On triangulated orbit categories}.  
Documenta Math. \textbf{10} (2005), 551--581.

\bibitem{Kel11-defoCY} 
B. Keller:
\emph{Deformed Calabi-Yau completions}.  
With an appendix by Michel Van den Bergh.  
J. Reine Angew. Math. \textbf{654} (2011), 125--180.
      
\bibitem{KeRe} 
B. Keller, I. Reiten:
\emph{Cluster-tilted algebras are Gorenstein and stably Calabi-Yau.} 
Adv. Math. \textbf{211} (2007), no 1, 123--151.

\bibitem{RIN} C.M. Ringel: 
\emph{Tame algebras and integral quadratic forms}. 
Lecture Notes in Mathematics, \textbf{1099}, Springer-Verlag, Berlin, 1984. 

\bibitem{Sko} A. Skowroński:
\emph{Selfinjective algebras of polynomial growth}, 
Math. Ann. \textbf{285} (1989), no. 2, 177--199.

\end{thebibliography}
\end{document}